\documentclass[a4paper,reqno,12pt]{amsart}
\usepackage{amssymb}
\usepackage{amsfonts}
\usepackage{bbm}
\usepackage{graphicx}
\usepackage{color}
\usepackage{epstopdf}
\usepackage{amsmath,amsthm}

\newtheorem{theorem}{Theorem}[section]
\theoremstyle{plain}

\newtheorem{corollary}[theorem]{Corollary}

\newtheorem{example}{Example}

\newtheorem{lemma}[theorem]{Lemma}

\newtheorem{proposition}[theorem]{Proposition}
\newtheorem{remark}[theorem]{Remark}

\renewcommand\bigskip{\medskip}

\def\to{\rightarrow}
\def\bc{\begin{center}}
\def\ec{\end{center}}
\def\be{\begin{equation}}
\def\ee{\end{equation}}
\def\P{\mathbb P}

\def\N{\mathbb N}

\def\R{\mathbb R}

\def\E{\mathbb E}
\def\pv{P_{\varphi}}

\begin{document}
\title[]{Multifractal analysis of some multiple ergodic averages}
\author{Ai-Hua FAN, J\"{o}rg SCHMELING and Meng WU}
\date{%
}
\address{A. F. Fan: LAMFA, UMR 7352 (ex 6140) CNRS, Universit\'e de Picardie, 33 rue
Saint Leu, 80039 Amiens, France.
E-mail: ai-hua.fan@u-picardie.fr}
\address{J. Scheming: MCMS,
Lund Institute of Technology, Lund University Box 118 SE-221 00
Lund, Sweden.
E-mail: joerg@maths.lth.se}
\address{M. Wu: LAMFA, UMR 7352 (ex 6140) CNRS, Universit\'e de Picardie, 33 rue
Saint Leu, 80039 Amiens, France.
E-mail: meng.wu@u-picardie.fr}

\maketitle

\begin{abstract}
In this paper we study the multiple ergodic averages $$
\frac{1}{n}\sum_{k=1}^n \varphi(x_k, x_{kq}, \cdots, x_{k q^{\ell-1}}),
\qquad (x_n) \in \Sigma_m
$$ on the
symbolic space $\Sigma_m =\{0, 1, \cdots, m-1\}^{\mathbb{N}^*}$ where $m\ge 2, \ell\ge 2, q\ge 2$ are integers. We give
 a complete solution to the problem of multifractal analysis of the limit of the above multiple ergodic averages.
  Actually we develop a non-invariant
and non-linear version of thermodynamic formalism that is of its own interest. We study a large class of measures (called telescopic measures) and the special case of telescopic measures defined by the fixed points of some non-linear transfer operators plays a crucial role in studying our multiplicatively invariant sets. These measures share many properties with Gibbs measures in the classical thermodynamic formalism. Our work also concerns with variational principle, pressure function and Legendre transform in this new setting.
\end{abstract}

\markboth{Multifractal analysis of some multiple ergodic averages}{Ai-Hua FAN, J\"{o}rg SCHMELING and Meng WU}%

\section{Introduction}\label{int}

Let $(X,T)$ be a topological dynamical system where $T$ is a
continuous map on  a compact metric space $X$. F\"{u}rstenberg had
initiated the study of the {\em multiple ergodic average}:
\begin{equation}\label{Furstenberg}
\frac{1}{n}\sum_{k=1}^nf_1(T^kx) f_2(T^{2k}x)\cdots f_s(T^{s k}x)
\end{equation}
where $f_1,\cdots ,f_s$ are $s$ continuous functions on $X$ with
$s\geq 2$ when he gave a proof of the existence of arithmetic sequences of arbitrary length amongst
sets of integers with positive density (\cite{Furstenberg}). Later on, the research of such a kind of average has
attributed  a lot of attentions (see e.g. \cite{Bergelson,Bour,Assani,HK}).

 The authors in \cite{FLM} have recently proposed to analyze such
 multiple ergodic averages from the point of view of
multifractal analysis. They have succeeded in a very special case
where $(X,T)$ is the shift dynamics on symbolic space and $f_1,
\cdots f_s$ are Rademacher functions on the symbolic space viewed
as an additive group. It is a challenge to solve the problem in its generality.

In the present paper, we shall consider the problem for the shift dynamics
and for a class of functions $f_1, \cdots, f_s$. The setting is
as follows. Let $S=\{0,\cdots,m-1\}$ be a set of $m$ symbols \
($m\geq 2$).  Consider the shift map $T$ on the symbolic space
$X=\Sigma_m=S^\N$. Fix two integers $q\geq 2$ and $\ell\ge 2$. For
any given $\ell$ continuous functions $g_1, g_2, \cdots, g_\ell$
defined on $X$,  we consider the multiple ergodic average
$$A_n(g_1,g_{2}\cdots, g_{\ell})(x)=\frac{1}{n}\sum_{k=1}^n
  g_1(T^kx)g_{2}(T^{kq}x)\cdots g_{\ell}(T^{kq^{\ell-1}}x).$$
This is a special case of (\ref{Furstenberg}) with $s=q^{\ell-1}$,
$f_{q^j} =g_{j-1}$ and $f_k =1$ for other $k \not=q^j$. Furthermore
we assume that the functions $f_1,f_{2}\cdots, f_{\ell}$ depend
 only on the first coordinate of $x=(x_k)_{k\ge 0} \in \Sigma_m$.
So, under this assumption of $f_j$'s we  have
\begin{equation}\label{MEA1}
A_n(g_1,g_{2}\cdots,
g_{\ell})(x)=\frac{1}{n}\sum_{k=1}^n
  g_1(x_k)g_{2}(x_{kq})\cdots g_{\ell}(x_{kq^{\ell-1}}).
  \end{equation}
  For the time being, there is no idea for the multifractal analysis of (\ref{Furstenberg})
  in its general form. So we are content with investigating
  the special case (\ref{MEA1}). Actually we can do a little more.
  Given a function $\varphi: S^\ell \to \mathbb{R}$ we shall study
\begin{equation}\label{GMEA}
A_n\varphi(x)=\frac{1}{n}\sum_{k=1}^n
  \varphi(x_k, x_{qk}, \cdots, x_{q^{\ell-1}}).
  \end{equation}
The average in (\ref{MEA1}) corresponds to the special case of (\ref{GMEA}) with $\varphi=g_1\otimes\cdots \otimes g_{\ell}$.
 For $\alpha\in \R$, we define
$$E(\alpha)=\left\{x\in \Sigma_m : \lim_{n\to\infty}A_n\varphi(x)=\alpha\right\}.$$
Our problem is to determine the Hausdorff dimension of $E(\alpha)$.
The problem is classical when $\ell=1$ and the answers are well
known (see e.g. \cite{FFW,FLP,BSS,Barreira}). Let
$$\alpha_{\min}=\min_{a_1,\cdots,a_{\ell}\in
S }\varphi(a_1,\cdots,a_{\ell}), \quad
\alpha_{\max}=\max_{a_1,\cdots,a_{\ell}\in
S}\varphi(a_1,\cdots,a_{\ell}).$$ We assume that
$\alpha_{\min}<\alpha_{\max}$ (otherwise $\varphi$ is constant and
the problem is trivial).

Let $\mathcal{F}(S^{\ell-1}, \mathbb{R}^+)$ be the cone of
functions defined on $S^{\ell-1}$ taking non-negative real values. For any $s \in
\mathbb{R}$, consider the  transfer operator $\mathcal{L}_s$ defined on
$\mathcal{F}(S^{\ell-1}, \mathbb{R}^+)$ by
\begin{equation}\label{transer-operator}
\mathcal{L}_s \psi (a)
= \sum_{j \in S} e^{s \varphi(a, j)}
\psi (Ta, j)
\end{equation}
where $T: \ S^{\ell-1}\to S^{\ell-2}$ is defined by $T(a_1,\cdots,a_{\ell-1})=(a_2,\cdots,a_{\ell-1})$.
We also consider the non-linear operator $\mathcal{N}_s$ on $\mathcal{F}(S^{\ell-1}, \mathbb{R}^+)$ defined by
$$
     \mathcal{N}_s \psi (a)= (\mathcal{L}_s \psi (a))^{1/q}.
$$


We shall
prove that the equation
\begin{equation}\label{transer_equation}
     \mathcal{N}_s \psi_s = \psi_s
\end{equation}
admits a unique strictly positive solution $\psi_s=\psi_s^{(\ell -1)} : S^{\ell-1}\to
\mathbb{R}_+^*$  (see Section \ref{section nonlinear equa}, Theorem \ref{existence-unicity trans-equ}).
The function $\psi_s$ is defined on $S^{\ell-1}$. We  extend it on $S^k$ for all $1\le k \le \ell -2$ by induction:

\begin{equation}\label{transer_equation 2}
\psi_s^{(k)} (a)=\left(\sum_{j \in S} \psi_s^{(k+1)} (a, j)\right)^{\frac{1}{q}}, \ \ (a\in S^{k}).
\end{equation}
For simplicity, we will simply write $\psi_s(a)=\psi_s^{(k)}(a)$ for $a\in S^k$ with $1\leq k\leq \ell-1$. So, $a\mapsto \psi_s(a)$ is not only defined on $S^{\ell-1}$ but on $\bigcup_{1\leq k\leq \ell-1}S^k$.

Then we define the pressure function by
\begin{equation}\label{pressure function}
P_{\varphi}(s) = (q-1)q^{\ell-2} \log \sum_{j\in S}\psi_s(j).
\end{equation}
Throughout this paper, $\log$ means the natural logarithm.

We will prove that $\pv(s)$ is an analytic convex function of $s\in
\mathbb{R}$ and even strictly convex since $\alpha_{\min}
<\alpha_{\max}$. The Legendre transform of $P_{\varphi}$ is defined
as $$ P^*_{\varphi}(\alpha)=\inf_{s\in \mathbb{R}}(-s\alpha+P_{\varphi}(s)).
$$

We denote by $L_{\varphi}$ the set of $\alpha\in \R$ such that
$E(\alpha)\neq \emptyset$. One of the main results of the paper is stated as follows.
    \medskip

\begin{theorem}\label{thm principal}
We have
$$L_{\varphi}=[P'_{\varphi}(-\infty),P'_{\varphi}(+\infty)].$$
 If
$\alpha =P'_{\varphi}(s_\alpha)$ for some $s_\alpha \in
\R\cup\{-\infty,+\infty\}$, then $E(\alpha)\neq \emptyset$ and the
Hausdorff dimension of $E(\alpha)$ is equal to
$$\dim_H E(\alpha)=\frac{P_{\varphi}^*(\alpha)}{q^{\ell-1}\log m}.$$
\end{theorem}
\medskip
This result was announced for $\ell=2$ in \cite{FSW}.
It is obvious that $L_{\varphi}\subset[\alpha_{\min},\alpha_{\max}]$. In general, this inclusion is strict. In fact, we have the following criterion for $L_{\varphi}=[\alpha_{\min},\alpha_{\max}]$.
\begin{theorem}\label{critere cercle}
We have the equality $$P'_{\varphi}(-\infty)=\alpha_{\min}$$ if and only if there exist an $x=(x_i)_{i=1}^\infty\in \Sigma_m$ such that $$\forall k\geq 1,\  \varphi(x_k,x_{k+1},\cdots,x_{k+\ell-1})=\alpha_{\min}.$$ We have analogue criterion for
$P'_{\varphi}(+\infty)=\alpha_{\max}.$
\end{theorem}

Let us look at the definition of
$$A_n\varphi(x) =\frac{1}{n}\sum_{k=1}^n
  \varphi(x_k, x_{kq}, \cdots, x_{k q^{\ell-1}}).
 $$
 One of the key points in our study of the problem is the observation that the coordinates $x_1,\cdots , x_n,\cdots $ of $x$ appearing in the definition of $A_n\varphi(x)$ share the following independence. This observation was first exploited in \cite{FLM} in order to compute the Box dimension of some subset of $E(\alpha_{\min})$. Consider the following partition of $\N^*$:
 $$
 \N^*=\bigsqcup_{i\geq 1,q\nmid i}\Lambda_i\ \ {\rm with}\ \Lambda_i=\{iq^j\}_{j\ge 0}.
 $$
 Observe that if $k=iq^j$ with $q\nmid i$, then $\varphi(x_k,x_{kq},\cdots ,x_{kq^{\ell-1}})$ depends only on $x_{|_{\Lambda_i}}$, the restriction of $x$ on $\Lambda_i$. So the summands in the definition of $A_n\varphi(x)$ can be put into different groups, each of which depends on one restriction $x_{|_{\Lambda_i}}$. For this reason, we decompose $\Sigma_m$ as follows:
 $$
 \Sigma_m=\prod_{i\geq1,q\nmid i}S^{\Lambda_i}.
 $$

 Let $\mu$ be a probability measure on $\Sigma_m$. Notice that $S^{\Lambda_i}$ is nothing but a copy of $\Sigma_m$.
 We consider $\mu$ as a measure on $S^{\Lambda_i}$ for every $i$ with $q\nmid i$. Then we define the infinite product measure $\P_\mu$ on   $\prod_{i\geq1,q\nmid i}S^{\Lambda_i}$ of the copies of $\mu$. More precisely,  for any word $u$ of length $n$ we define
 $$
 \P_{\mu}([u])=\prod_{i\leq n,q\nmid i}\mu([u_{|_{\Lambda_i}}]),
 $$
 where $[u]$ denotes the cylinder of all sequences starting with $u$. Then $\P_{\mu}$ is a probability measure on $\Sigma_m$ and we call it a {\em telescopic product measure}.  Kenyon, Peres and Solomyak \cite{KPS,KPS1} used this kind of measures to compute the Hausdorff dimension of sets like $\{x=(x_n)_{n\ge 1}\in \Sigma_2: \forall k\ge 1, x_kx_{2k}=0\}$
 which was proposed in \cite{FLM}.
 \medskip

 A class of measures $\P_\mu$ will play the same role as Gibbs measures played in the study of simple ergodic  averages ($\ell=1$).
 Concerning the dimension of $\P_\mu$ (see \cite{Fan1994} for the dimension of a measure), we have the following result
 which is one of the main ingredients of the proof of the main result (Theorem \ref{thm principal}) and which has its own interest. A measure $\nu$ on $\Sigma_m$ is said to be exact if there exists an $\alpha\in \R$ such that
 $$ \lim_{n\to \infty}\frac{\log_m\nu ([x_{|_n}])}{n}=\alpha,\ \nu{\rm -a. e.}$$ This value $\alpha$ is the dimension of $\nu$.

 \begin{theorem}\label{dim-meas} For any given measure $\mu$,
the telescopic product measure $\mathbb{P}_\mu$ is exact and its dimension is equal to
  $$\dim_{H}\P_{\mu}=\frac{(q-1)^2}{\log m}\sum_{k=1}^{\infty}\frac{H_k(\mu)}{q^{k+1}}$$
  where
  $$H_k(\mu)=-\sum_{a_1,\cdots,a_k\in S}\mu([a_1\cdots a_k])\log \mu([a_1\cdots a_k]).
  $$
\end{theorem}
A similar formula for some special $\P_{\mu}$ has appeared in \cite{KPS1}.
Another ingredient of the proof of Theorem \ref{thm principal} is a law of large numbers relative to the probability
$\P_\mu$.
We consider $(\prod_{i\geq1,q\nmid i}S^{\Lambda_i},\P_\mu)$ as a probability space $(\Omega,\P_\mu)$. Let $(F_k)_{k\ge 1}$ be a sequence of functions defined on $\Sigma_m$. For each $k$, there exists a unique integer $i(k)$ such that $k=i(k)q^j$ and $q\nmid i$. Then
$$
x \mapsto F_k(x_{|_{\Lambda_{i(k)}}})
$$
defines a random variable on $\Omega$. Concerning the sequence of random variables $\left\{F_k(x_{|_{\Lambda_{i(k)}}})\right\}$, we have the following law of large numbers.

\begin{theorem}\label{LLN}
Let $(F_k)_{k\ge 1}$ be a sequence of functions defined on $\Sigma_m$. Suppose that there exist \  $C>0$ and $0<\eta<q^{3/2}$ such that for any $i\geq 1$ with $q\nmid i$, any $j_1,j_2\in \N$, we have
\begin{equation}\label{condition2}
{\rm cov}_{\mu}\left(F_{iq^{j_1}}(x), F_{iq^{j_2}}(x)\right)\leq C\eta^{\frac{j_1+j_2}{2}}.
\end{equation}
Then for  $\P_{\mu}-$a.e. $x\in \Sigma_{m}$
$$\lim_{n\to\infty}\frac{1}{n}\sum_{k=1}^{n}\left(F_k(x_{|_{\Lambda_{i(k)}}})-\E_{\mu}F_k(x)\right)=0. $$
\end{theorem}

We observe that the set $E(\alpha)$ is not invariant. So it is not a standard set studied
from the classical dynamical system point of view. Actually, as we shall see, in general the dimension of the set $E(\alpha)$  can not be
described by invariant measures supported on  it. This is confirmed by the following result.
\medskip

Given two real valued functions $f_1$ and $f_2$ defined on $\Sigma_m$. For $\alpha\in \mathbb{R}$, let $E(\alpha)$ be
the set of all points $x$ such that
$$
    \lim_{n \to \infty}\frac{1}{n}\sum_{k=1}^n f_1(T^k x) f_2(T^{2k}x) = \alpha.
$$
We describe the size of the invariant part of $E(\alpha)$ by
$$
    F_{\rm inv}(\alpha) = \sup \left\{\dim \mu: \mu \ \mbox{\rm ergodic}, \mu(E(\alpha))=1\ \right\}.
$$

\begin{theorem}\label{invariant} Let $f_1$ and $f_2$ be two H\"{o}lder continuous functions on $\Sigma_m$.
If $E(\alpha)$ supports an ergodic measure, then
$$
F_{\rm inv}(\alpha)
= \sup \left\{\dim \mu: \mu \ \mbox{\rm ergodic}, \int f_1 d\mu \int f_2 d\mu = \alpha \ \right\}.
$$
    \end{theorem}
It is interesting to compare this result with the level sets of $V$-statistics studied in \cite{FSW_V}.
We return to the above theorem. A remarkable corollary is that when $f_1=f_2$, we must have $\alpha\ge 0$
if $E(\alpha)$ supports an ergodic measure, or even an invariant measure (using Jacobs' entropy decomposition).
Therefore, it is possible that for some $\alpha<0$, $E(\alpha)$ has strictly positive Hausdorff dimension but it doesn't carry  any invariant measure.

    \medskip
The paper is organized as follows. In Section 2, we first construct a class of measures, called telescopic product measures, part of which will play the same role
as Gibbs measures played in the classical theory. This construction is inspired by Kenyon-Peres-Solomyak \cite{KPS} (also see \cite{KPS1}). Then we establish a law of large numbers relative to such a telescopic product measure.
Telescopic product measures constitute a new object of study. In Section 3, we prove that any telescopic product measure is exact and we obtain
a formula for its dimension. In Section 4, we study a non--linear transfer operator and we prove the existence and the uniqueness of its positive solution.
We also prove the analyticity and the convexity of the solution as a function of its parameter $s$.  Each solution defines a Markov measure associated to which is a telescopic product measure. The last measure plays the role of a Gibbs measure in our study of $E(\alpha)$.
Section 5 is devoted to the properties of the pressure function: a Ruelle type formula says that the limit in the law of large numbers is the derivative
of the pressure; the pressure function is  an analytic and strictly convex function (except the trivial case); the extreme values of the derivative of the
pressure are studied. In Section 6, we establish the Gibbs property of the telescopic product measures defined by the solution of the non--linear transfer operator.  After all these preparations, many of which have their own interests, we prove the main theorem (Theorem~\ref{thm principal}) in Section 7.
In Section 8, we discuss the invariant part of $E(\alpha)$.
Some concrete examples are presented in Section 9.
In the final section, we make some remarks and present some unsolved problems.

\medskip
{\em Acknowledgement:} The authors would like to thank B. Host for his interests in the work and especially for his remarks, some of which are contained
in Section 8.
\medskip

\section{Telescopic product measures and LLN \label{section2}}
In this section, we will study telescopic product measures and establish a law of large numbers (LLN).
These measures, which take into account the multiplicative structure of the multiple ergodic averages $A_n\varphi(x)$, will play
the same role as Gibbs measures played in the study of simple ergodic averages. In the next section,  we will prove that $\mathbb{P}_\mu$ is
exact and its dimension is equal to
  $$\dim_{H}\P_{\mu}=\frac{(q-1)^2}{\log m}\sum_{k=1}^{\infty}\frac{H_k(\mu)}{q^{k+1}}$$
  where
  $$H_k(\mu)=-\sum_{a_1,\cdots,a_k\in S}\mu([a_1\cdots a_k])\log \mu([a_1\cdots a_k]).
  $$
  We could call $H_k$ the $k$-th entropy of $\mu$. But we should point out that $\mu$
  is not assumed to be invariant and that $\mathbb{P}_\mu$ is not invariant either.

\subsection{Telescopic product measures}
Let us recall the definition of the telescopic product measure $\P_\mu$.  Consider the following partition of $\N^*$:
 $$
 \N^*=\bigsqcup_{i\geq 1,q\nmid i}\Lambda_i\ \ {\rm with}\ \Lambda_i=\{iq^j\}_{j\ge 0}.
 $$
 Then we decompose $\Sigma_m$ as follows:
 $$
 \Sigma_m=\prod_{i\geq1,q\nmid i}S^{\Lambda_i}.
 $$

 Let $\mu$ be a probability measure on $\Sigma_m$.
 We consider $\mu$ as a measure on $S^{\Lambda_i}$, which is identified with $\Sigma_m$, for every $i$ with $q\nmid i$. Then we define the infinite product measure $\P_\mu$ on   $\prod_{i\geq1,q\nmid i}S^{\Lambda_i}$ of the copies of $\mu$. More precisely,  for any word $u$ of length $n$ we define
 $$
 \P_{\mu}([u])=\prod_{i\leq n,q\nmid i}\mu([u_{|_{\Lambda_i}}]),
 $$
 where $[u]$ denotes the cylinder of all sequences starting with $u$.
 \medskip

We consider $(\Sigma_{m},\P_{\mu})$ as a probability space. Let $X_{k}(x)=x_{k}$ be the $k$-th coordinate projection. For each $i$ with $q\nmid i$, consider the process $Y^{(i)}=(X_k)_{k\in \Lambda_i}$. Then, by the definition of $\P_{\mu}$, the following fact is obvious.
\begin{lemma}\label{prop ind}
The processes $Y^{(i)}=(X_k)_{k\in \Lambda_i}$ for different $i\geq 1 $ with $q\nmid i$ are $\P_{\mu}$-independent and identically distributed with
 $\mu$ as the common probability law.
\end{lemma}

As we shall see, the behaviour of $A_n\varphi(x)$ as $n \to \infty$ will be described by measures $\P_{\mu}$ with particular choices of $\mu$. It is natural that $\P_\mu$ strongly depends on the above partition of $\N^*$. The following is a detail of the partition which will be useful.
Fix $n\in \N^*$. Let $$\Lambda_i(n)=\Lambda_i\cap \{1,\cdots,n\}. $$
We are going to examine the cardinality $\sharp\Lambda_i(n)$, called the length of $\Lambda_i(n)$ and  the number $N(n,q,k)$ of $\Lambda_i(n)$'s of a given length $k$.

\begin{lemma}\label{lem dec}\ Let $k, n \in \N^*$.\\
\indent {\rm
(1)}  $\sharp \Lambda_i(n)=k$ if and only if $\frac{n}{q^k}<i\leq \frac{n}{q^{k-1}}.$ Consequently we have $$\sharp \Lambda_i(n)=\left\lfloor\log_q\frac{n}{i}\right\rfloor.$$ \\
\indent {\rm (2)} We have the partition
$$\{1,\cdots,n\}=\bigsqcup_{k=1}^{\lfloor \log_q
n\rfloor}\ \ \bigsqcup_{\frac{n}{q^k} <i\leq \frac{n}{q^{k-1}}, q\nmid
i}\Lambda_i(n).$$
\indent {\rm  (3)}  $N(n,q,k)$ is the number of $i$'s such that $q\nmid i$ and  $\frac{n}{q^k} <i\leq \frac{n}{q^{k-1}}$. 
We have
$$ \left| \frac{N(n,q,k)}{n}-\frac{(q-1)^2}{q^{k+1}}\right| \le \frac{4}{n}.$$
\end{lemma}

\begin{proof}
(1) It is simply because $\sharp \Lambda_i(n)=k$ means that $$\Lambda_{i}(n)=\{i,iq,\cdots,iq^{k-1}\}\  { \rm with }\ iq^{k-1}\leq n<iq^{k}.$$\\
(2) We have the obvious partition  $$\{1,\cdots , n\}=\bigsqcup_{i\leq n,q\nmid i}\Lambda_{i}(n).$$ Then we collect $\Lambda_i(n)$ by their lengths. By (1), we have  $1\leq\sharp\Lambda_{i}(n)\leq \left\lfloor\log_{q}n\right\rfloor$ and
$$\{1,\cdots , n\}=\bigsqcup_{k=1}^{\left\lfloor\log_{q}n\right\rfloor}\bigsqcup_{\substack{i<n,q\nmid i, \\
\sharp\Lambda_{i}(n)= k}}\Lambda_{i}(n).$$
(3) By (1), $N(n, q, k)$ is obviously the numbers of $i$ such that $\frac{n}{q^k}< i \le \frac{n}{q^{k-1}}$ and $q\nmid i$.
It is the number of $i$'s such that $\frac{n}{q^k}< i \le \frac{n}{q^{k-1}}$ minus the $i$'s such that $\frac{n}{q^k}< i \le \frac{n}{q^{k-1}}$
and $q\mid i$, 
i.e.
$$
    N(n,q,k) = \left( \left\lceil \frac{n}{q^{k-1}}\right\rceil - \left\lceil \frac{n}{q^{k}}\right\rceil\right)
    - \left( \left\lceil \frac{n}{q^{k}}\right\rceil - \left\lceil \frac{n}{q^{k+1}}\right\rceil\right).
$$
 It follows that
$$
\left| N(n,q,k)- \left(\frac{n}{q^{k-1}}- \frac{2n}{q^k} + \frac{n}{q^{k+1}}\right)\right| \le 4. $$
It is the desired estimate for $\frac{1}{q^{k-1}}-\frac{2}{q^k} +\frac{1}{q^{k+1}} = \frac{(q-1)^2}{q^{k+1}}$.
\end{proof}

\bigskip

Now we consider $(\prod_{i\geq1,q\nmid i}S^{\Lambda_i},\P_\mu)$ as a probability space $(\Omega,\P_\mu)$. Let $(F_k)_{k\ge 1}$ be a sequence of functions defined on $\Sigma_m$. For each $k$, there exists a unique integer $i(k)$ such that $k=i(k)q^j$ and $q\nmid i$. Then $x \mapsto F_k(x_{|_{\Lambda_{i(k)}}})$
defines a random variable on $\Omega$. Later,
we will study the law of large numbers for the sequence of variables $\{F_k(x_{|_{\Lambda_{i(k)}}})\}_{k\ge 1}$.
Notice that if $i(k)\not=i(k')$, then the two variables $F_k(x_{|_{\Lambda_{i(k)}}})$ and $F_{k'}(x_{|_{\Lambda_{i(k')}}})$
are independent. But if $i(k)=i(k')$, they are not independent in general.
In order to prove the law of large numbers, we will need the following technical lemma which allows us to compute the expectation of the product of $F_k(x_{|_{\Lambda_{i(k)}}})$'s. The proof of the lemma is based on the independence of $x_{|_{\Lambda_i}}$'s.

\begin{lemma} \label{LLN lemma 0}
Let $(F_k)_{k\geq 1}$ be a sequence of functions defined on $\Sigma_m$. Then for any integer $N\geq 1$, we have
$$\E_{\P_\mu}\left(\prod_{k=1}^N F_k(x_{|_{\Lambda_{i(k)}}})\right)=\prod_{k=1}^{\lfloor\log_q N\rfloor}\prod_{\frac{N}{q^k} <i\leq \frac{N}{q^{k-1}}, q\nmid i}\E_{\mu}\left(\prod_{h=0}^{k-1}F_{iq^h}(x)\right).$$
In particular, 
for any function $G$ defined on $\Sigma_m$, for any $n\geq 1$,
$$\E_{\P_\mu}G(x_{|_{\Lambda_{i(n)}}})=\E_{\mu}G(\cdot).$$
\end{lemma}

\begin{proof}
Let $$Q_{N}(x)=\prod_{k=1}^NF_k(x_{|_{\Lambda_{i(k)}}}), \ \ Q_{N,i}(x)=\prod_{k\in \Lambda_{i}(N)}F_k(x_{|_{\Lambda_i}}).$$
Since the variables $x_{|_{\Lambda_{i}}}$ for different $i\geq1 $ with $q\nmid i$ are independent under $\P_{\mu}$ (by Lemma \ref{prop ind}), we have
\begin{equation}\label{1}
\E_{\P_{\mu}}Q_{N}=\prod_{i\leq N,q\nmid i}\E_{\P_{\mu}}Q_{N,i}.
\end{equation}
 Then, by (2) of Lemma \ref{lem dec}, we can rewrite the right hand side in (\ref{1}) to get
 $$\E_{\P_{\mu}}Q_{N}=\prod_{k=1}^{\lfloor\log_q N\rfloor}\prod_{\frac{N}{q^k} <i\leq \frac{N}{q^{k-1}}} \E_{\P_{\mu}}Q_{N,i}.$$
However, the marginal measures on $S^{\Lambda_i}$ of $\P_{\mu}$ is equal to $\mu$ and $\Lambda_i(N)=\{i,iq,\cdots,iq^{k-1}\}$ if $\frac{N}{q^k} <i\leq \frac{N}{q^{k-1}}$.  So
 $$\E_{\P_{\mu}}Q_{N,i}=\E_{\mu}\left(\prod_{h=0}^{k-1}F_{iq^h}(x)\right).$$

 Now, for any function $G$ defined on $\Sigma_m$ and any $n\in \N^*$, if we set $F_n=G$ and $F_k=1$ for $k\neq n$ we have
 $$\E_{\P_\mu}G(x_{|_{\Lambda_{i(n)}}})=\E_{\mu}G(x).$$
\end{proof}

\subsection{Law of large numbers}
In order to prove the law of large numbers (LLN), we need  the following result.

Recall that the covariance of two bounded functions $f,g$ with respect to $\mu$ is defined by
$${\rm cov}_\mu(f,g)=\E_{\mu}\left[(f-\E_\mu f)(g-\E_\mu g)\right]$$
\begin{proposition}\label{lemma LLN}
Let $(F_k)_{k\ge 1}$ be a sequence of functions defined on $\Sigma_m$ satisfying
\begin{equation}\label{condition1}
{\rm cov}_\mu \left(F_{iq^{j_1}}(x)F_{iq^{j_2}}(x)\right)\leq C\eta^{\frac{j_1+j_2}{2}}
\end{equation}
  for some constants $C>0$ and $0<\eta<q^{3/2}$ and for all $i \ge 1$ with $q\nmid i$ and all $j_1, j_2 \in \mathbb{N}$.
  Let $p_0, p_1$ and $p_2$ be three maps from $\N^*$ into $\N^*$ such that
\begin{equation}\label{condition3}
\forall n\in \N^*, \ \ 1\le \frac{p_2(n)}{p_1(n)}\leq \alpha; \quad
\sum_{n=1}^\infty \frac{p_2(n)^{\frac{3}{2}-\epsilon}}{p_0(n)^2}<+\infty.
\end{equation}
for some $\alpha>1$ and some $0<\epsilon<1/2$ with $q^{3/2-\epsilon}>\eta$.
Then for  $\P_{\mu}-$a.e. $x\in \Sigma_{m}$
$$\lim_{n\to\infty}\frac{1}{p_0(n)}\sum_{k=p_1(n)}^{p_2(n)}\left(F_k(x_{|_{\Lambda_{i(k)}}})-\E_{\mu}F_k(x)\right)=0. $$
\end{proposition}

\begin{proof}  Without loss of generality, we can assume that $\E_{\P_\mu}F_k(x_{|_{\Lambda_{i(k)}}})=0$ for all $k\in \N^*$. Otherwise, we replace $F_k(x_{|_{\Lambda_{i(k)}}})$ by $F_k(x_{|_{\Lambda_{i(k)}}})-\E_{\P_\mu}F_k(x_{|_{\Lambda_{i(k)}}})$. 
We denote
$$
Z_n=\frac{1}{p_0(n)} \sum_{k=p_1(n)}^{p_2(n)} Y_k \quad {\rm with}\ \ Y_k=F_k(x_{|_{\Lambda_{i(k)}}}).
$$
We have only to show that
$$
\sum_{n=1}^{\infty}\E_{\P_\mu}Z_{n}^2<+\infty.
$$
Notice that $$
\E_{\P_\mu}Z_{n}^2=\frac{1}{p_0^2(n)}
 \sum_{p_1(n)\leq
u,v\leq p_2(n)}\E_{\P_\mu} Y_uY _v.
$$
Observe that by Lemma \ref{prop ind}, $\E_{\P_\mu}Y_uY_v\neq 0$ only if $i(u)=i(v)$, in other words only if
$u$ and $v$ are in the same set $\Lambda_i$. So
\begin{equation}\label{LLN lemma 1}
\E_{\P_\mu}Z_{n}^2=\frac{1}{p_0^2(n)} \sum_{\substack{i\geq1,q\nmid i, \\
\Lambda_i\cap[p_1(n),p_2(n)]\neq \emptyset}}\sum_{u,v\in\Lambda_i\cap[p_1(n),p_2(n)]}\E_{\P_\mu} Y_uY_v.
\end{equation}
However by the hypothesis (\ref{condition1}) on  the sequence $(F_k)_{k\ge 1}$, for any $u,v\in\Lambda_i\cap[p_1(n),p_2(n)]$ we have
$$\left|\E_{\P_\mu} Y_uY_v\right|=\left|\E_{\mu}F_u(x)F_v(x)\right|\leq C\eta^{\log_q \frac{p_2(n)}{i}}.$$

Substituting the last estimate into (\ref{LLN lemma 1}), we get
\begin{equation}\label{LLN lemma 2}
\E_{\P_\mu}Z_{n}^2\leq\frac{C}{p_0^2(n)}\sum_{\substack{i\geq1,q\nmid i, \\
\Lambda_i\cap[p_1(n),p_2(n)]\neq \emptyset}}\eta^{\log_q \frac{p_2(n)}{i}}\sharp \left(\Lambda_i\cap[p_1(n),p_2(n)]\right).
\end{equation}
The cardinality $\sharp \left(\Lambda_i\cap[p_1(n),p_2(n)]\right)$ is estimated as follows:
\begin{equation}\label{LLN lemma 3}
\sharp \left(\Lambda_i\cap[p_1(n),p_2(n)]\right)\leq 1+\log_q\alpha.
\end{equation}
In fact, assume that $$\Lambda_i\cap[p_1(n),p_2(n)]=\{a_1,\cdots,a_k\}$$
with $a_1<\cdots <a_k$.
 Then by the definition of $\Lambda_i$, we must have $\frac{a_{j+1}}{a_j}\geq q$ for $1\leq j\leq k-1$ so that $$\frac{a_k}{a_1}\geq q^{k-1}.$$ On the other hand,
$$\frac{a_k}{a_1}\leq \frac{p_2(n)}{p_1(n)}\leq \alpha. $$
So $q^{k-1}\leq \alpha$, i.e. $k \le 1+\log_q\alpha$.
Substituting (\ref{LLN lemma 3}) into (\ref{LLN lemma 2}), we get
\begin{equation}\label{LLN lemma 4}
\E_{\P_\mu}Z_{n}^2\leq\frac{C(1+\log_q\alpha)^+}{p_0^2(n)}\sum_{\substack{i\geq1,q\nmid i, \\
\Lambda_i\cap[p_1(n),p_2(n)]\neq \emptyset}}\eta^{\log_q \frac{p_2(n)}{i}}.
\end{equation}

There are at most  $p_2(n)-p_1(n)$ integers $i$ such that $i\geq1,q\nmid i$ and $\Lambda_i\cap[p_1(n),p_2(n)]\neq \emptyset$.
 If they are increasingly ordered, then the $j$-th is bigger than $j$. We deduce that
$$\sum_{\substack{i\geq1,q\nmid i, \\
\Lambda_i\cap[p_1(n),p_2(n)]\neq \emptyset}}\eta^{\log_q \frac{p_2(n)}{i}}
\leq \sum_{j=1}^{p_2(n)-p_1(n)}\eta^{\log_q \frac{p_2(n)}{j}} \le \sum_{j=1}^{p_2(n)-p_1(n)}\left(\frac{p_2(n)}{j}\right)^{\frac{3}{2}-\epsilon},
$$
where the last inequality is due to the fact that $\log_q \eta <3/2 -\epsilon$.
Since $\epsilon<1/2$, we have $\sum_{j=1}^\infty j^{-(3/2-\epsilon)}<\infty$.
Then
 $$ \E_{\P_\mu}Z_{n}^2\le C \frac{p_2(n)^{3/2-\epsilon}}{p_0(n)^2}.$$
 We conclude by the hypothesis which says that the right hand side of the above estimate is the general term of a
 convergent series.
\end{proof}

The following is the LLN which will be useful for our computation of the dimension of the telescopic
product measure $\mathbb{P}_\mu$.

\begin{theorem}\label{LLN}
Let $(F_k)_{k\ge 1}$ be a sequence of functions defined on $\Sigma_m$. Suppose that there exist \  $C>0$ and $0<\eta<q^{3/2}$ such that for any $i\geq 1$ with $q\nmid i$, any $j_1,j_2\in \N$,
\begin{equation}\label{condition2}
{\rm cov}_{\mu}\left(F_{iq^{j_1}}(x), F_{iq^{j_2}}(x)\right)\leq C\eta^{\frac{j_1+j_2}{2}}.
\end{equation}
Then for  $\P_{\mu}-$a.e. $x\in \Sigma_{m}$
$$\lim_{n\to\infty}\frac{1}{n}\sum_{k=1}^{n}\left(F_k(x_{|_{\Lambda_{i(k)}}})-\E_{\mu}F_k(x)\right)=0. $$
\end{theorem}

\begin{proof}
Without loss of generality, we can assume that $\E_{\P_\mu}F_k(x_{|_{\Lambda_{i(k)}}})=0$ for all $k\in \N^*$.
Our aim is to prove $\lim_{n\to\infty}Y_{n}=0$  $\P_{\mu}$-a.e., where
$$Y_{n}=\frac{1}{n}\sum_{k=1}^n X_k \ \ \ {\rm with} \ \ \ X_k= F_{k}(x|_{\Lambda_{i(k)}}).$$

First we claim that it suffices to show
\begin{equation}\label{LLN 1}
\lim_{n\to\infty}Y_{n^2}=0, \ \  \P_{\mu}- {\rm a.e. }
\end{equation}
In fact, for every $n\in \N $ there exists a unique $k\in \N$ such that $k^2\leq n<(k+1)^2$. Then we have
$$
\left|Y_{n}\right|\leq \left|Y_{k^2}\right|+\frac{\left(|X_{k^2+1}|+\cdots+|X_{n}|+\cdots+|X_{(k+1)^2}|\right)}{k^2}.
$$
So, since $Y_{k^2} \to 0$ $\mathbb{P}_\mu$-a.e.,
 we have only to show
\begin{equation}\label{LLN 3}
\lim_{k\to\infty}\frac{\left(|X_{k^2+1}|+\cdots+|X_{n}|+\cdots+|X_{(k+1)^2}|\right)}{k^2}=0, \ \  \P_{\mu}- {\rm a.e. }
\end{equation}
Let $p_0, p_1$ and $p_2$ be the three maps from $\N^*$ to $\N^*$ defined as follows:
$$p_0(k)=p_1(k)=k^2,\ \  p_2(k)=(k+1)^2 \ { \rm for }\ k\in \N^*.$$
Then observe that $$1\le \frac{p_2(k)}{p_1(k)}= \frac{(k+1)^2}{k^2}\leq 4 \ \ \forall  k\in \N^*$$
$$\sum_{k=2}^\infty \frac{(p_2(k)^{\frac{3}{2}-\epsilon})}{p_0^2(n)}\leq\sum_{k=2}^\infty \frac{((k+1)^2)^{\frac{3}{2}-\epsilon}}{k^4}<+\infty.$$
Thus we have verified that the maps $p_0, p_1$ and $p_2$ satisfy the hypothesis  of  Lemma \ref{lemma LLN}. Then (\ref{LLN 3}) is assured  by Proposition \ref{lemma LLN}.

Now we are going to show
\begin{equation}\label{LLN 4}
\sum_{n=1}^{\infty}\E_{\P_\mu}Y_{n^2}^2<+\infty,
\end{equation}
which will imply (\ref{LLN 1}).  Notice that
$$
\E_{\P_\mu}Y_{n}^2=\frac{1}{n^2}\sum_{1\leq
u,v\leq n}E_{\P_\mu}X_u X_v.
$$
By Lemma \ref{prop ind}, we have $\mathbb{E}_{\P_\mu}X_u X_v\neq 0$ only if $i(u)=i(v)$. So
$$
\E_{\P_\mu}Y_{n}^2=\sum_{i\leq n,q\nmid i}\ \sum_{u,v\in \Lambda_i(n)}\mathbb{E}_{\P_\mu}X_u X_v.
$$
By (2) of Lemma \ref{lem dec}, we can rewrite the above sum as
\begin{equation}\label{LLN 5}
\E_{\P_\mu}Y_{n}^2=\sum_{k=1}^{\lfloor\log_q n\rfloor}\sum_{\substack{\frac{n}{q^k} <i\leq \frac{n}{q^{k-1}} \\
q\nmid i}} \sum_{u,v\in \Lambda_i(n)}\E_{\P_\mu} X_u X_v.
\end{equation}
Recall that $\E_{\P_\mu}X_k =\E_{\mu}F_k$ for all $k\in\N^*$ (Lemma \ref{LLN lemma 0}).
For $u,v\in \Lambda_i(n)$, we write $u=iq^{j_1}$ and $v=iq^{j_2}$ with $0\leq j_1,j_2\leq \sharp\Lambda_i(n)$. By  the  Cauchy-Schwarz inequality and the hypothesis (\ref{condition2}), we obtain
$$\left|\E_{\mathbb{P}_\mu} X_uX_v\right|\leq \sqrt{\E_{\mu} F_u^2} \sqrt{\E_{\mu} F_v^2}\leq C\eta^{\sharp \Lambda_i(n)}.$$
This estimate holds for all $u,v\in\Lambda_i(n)$. So
$$\sum_{u,v\in \Lambda_i(n)}|\mathbb{E}_{\P_\mu}X_u X_v |\leq C\left(\sharp\Lambda_i(n)\right)^2\eta^{\sharp \Lambda_i(n)}.$$
Substituting this estimate into (\ref{LLN 5}) and using (1) of Lemma \ref{lem dec}, we get
$$
\left|\E_{\P_\mu}Y_n^2\right|\leq \frac{C}{n^2}\sum_{k=1}^{\lfloor\log_q n\rfloor}\sum_{\substack{\frac{n}{q^k} <i\leq \frac{n}{q^{k-1}} \\ q\nmid i}}k^2\eta^{k} =\frac{C}{n^2}\sum_{k=1}^{\lfloor\log_q n\rfloor}k^2\eta^{k}N(n,q,k),
$$
where $N(n,q,k)$ appeared in  Lemma \ref{lem dec}. Then by (3) of Lemma  \ref{lem dec}, the last term is equivalent to
$$\frac{C(q-1)^2}{n}\sum_{k=1}^{\lfloor\log_q n\rfloor}\frac{k^2 \eta^{k}}{q^{k+1}}= O\left(\frac{1}{n}\left(\frac{\eta}{q}\right)^{\log_q n}\right)
=O\left(n^{-1/2-\epsilon}\right)$$
for some $\epsilon>0$. This implies (\ref{LLN 4}).
\end{proof}

\medskip

\subsection{A special LLN}

When, in the LLN (Theorem \ref{LLN}), the functions $(F_i)_i$ are all  the same function $F$, then we have the following special LLN.

\begin{theorem}\label{thm esperence general formula}
Let $\mu$ be any probability measure $\mu$ on $\Sigma_m$ and let
$F\in \mathcal{F}(S^{\ell})$.  For $\P_{\mu}$ a.e. $x\in \Sigma_m$
we have
$$\lim_{n\to\infty}\frac{1}{n}\sum_{k=1}^nF(x_k,\cdots,x_{kq^{\ell-1}})=(q-1)^2\sum_{k=1}^\infty\frac{1}{q^{k+1}}\sum_{j=0}^{k-1}\E_\mu F(x_j,\cdots,x_{j+\ell-1}).$$
\end{theorem}

\begin{proof}
For any integer $k$ we write  $k=i(k)q^j$  with $q\nmid i(k)$. Then
we define a function $F_{k}$ by
$$F_{k}(x)=F(x_j,\cdots,x_{j+\ell-1}).$$
Therefore we can re-write
$$
F(x_k,x_{kq},\cdots,x_{kq^{\ell-1}})=F_k(x_{|_{\Lambda_{i(k)}}}).
$$
By the law of large numbers, for $\P_{\mu_s}$ a.e. $x\in\Sigma_m$ we
have
$$
\lim_{n\to\infty}\frac{1}{n}\sum_{k=1}^nF_k(x_{|_{\Lambda_{i(k)}}})
=\lim_{n\to\infty}\frac{1}{n}\sum_{k=1}^n\E_{\mu}F_k(x)
$$
if the limit in the right hand side exists. The limit does exists.
In fact,
 by (2) of Lemma \ref{lem dec}, we have
$$\sum_{k=1}^n\E_{\mu}F_k(x)=\sum_{k=1}^{\lfloor\log_q n\rfloor}\sum_{\substack{\frac{n}{q^k} <i\leq \frac{n}{q^{k-1}}\\ q\nmid i}}\sum_{j=0}^{\sharp \Lambda_i(n)-1}\E_{\mu}F_{iq^j}(x).$$
By the definition of the sequence $(F_k)$, for any $k=iq^j$  with
$q\nmid i$ we have
$$\E_{\mu}F_{iq^j}(x)=\E_{\mu}F(x_j,\cdots,x_{j+\ell-1}),$$
which is independent of $i$.
 Combining the last two equations, we
get
$$\sum_{k=1}^n\E_{\mu}F_k(x)=\sum_{k=1}^{\lfloor\log_q n\rfloor}N(n,q,k)\sum_{j=0}^{k-1}\E_{\mu}F(x_j,\cdots,x_{j+\ell-1}),$$
where $N(n,q,k)$ appeared in Lemma \ref{lem dec}. Then, by (3) of Lemma \ref{lem dec}, we get

\begin{eqnarray*}
\lim_{n\to\infty}\frac{1}{n}\sum_{k=1}^n\E_{\mu}F_k(x) & = &\lim_{n\to\infty}\sum_{k=1}^{\lfloor\log_q n\rfloor}\frac{N(n,q,k)}{n}\sum_{j=0}^{k-1}\E_{\mu}F(x_j,\cdots,x_{j+\ell-1})\\
&=& (q-1)^2\sum_{k=1}^\infty\frac{1}{q^{k+1}}\sum_{j=0}^{k-1}\E_\mu F(x_j,\cdots,x_{j+\ell-1}).
\end{eqnarray*}

\end{proof}

\medskip
\medskip

\section{Dimensions of telescopic product measures \label{section2+}}

Let $\nu$ be a measure on $\Sigma_m$. The lower local dimension of $\nu$ at a point $x\in \Sigma_m$ is defined as
$$\underline{D}(\nu,x):=\liminf_{n\to\infty}\frac{-\log_m\nu([x_1^n])}{n}.$$
Similarly, we can define the upper local dimension $\overline{D}(\nu,x)$. If $\underline{D}(\nu,x)=\overline{D}(\nu,x)$, we write $D(\nu, x)$ for the common value and
we say that $\nu$ admits $D(\nu, x)$ as the exact local dimension at $x$.
See \cite{Fan1994} for the dimensions of measures. Recall that the Hausdorff dimension of a Borel measure $\nu$, denoted by $\dim_{H}\nu$, is the minimal dimension of  Borel sets of full measure
 and is equal to ${\rm ess \ sup}_\nu \underline{D}(\nu,x)$ (\cite{Fan1994}). 

In this section, as a consequence of the LLN,
we will prove that  every telescopic product measure $\mathbb{P}_\mu$ admits its exact local dimension for $\P_\mu$-a.e. point in $\Sigma_m$, which is a constant.

\subsection{Local dimension of telescopic product measures}

For a measure $\mu$ on $\Sigma_m$ and for $k\geq1$, we define $$H_k(\mu)=-\sum_{a_1,\cdots,a_k}\mu([a_1\cdots a_k])\log \mu([a_1\cdots a_k]).$$
We note that for a probability measure $\mu$ we have $0\le H_k(\mu)\le k\log m$.

\begin{theorem}\label{prop loc dim}
For $\P_\mu$-a.e. $x\in \Sigma_m$, we have
$$D(\P_\mu,x)=\frac{(q-1)^2}{\log m}\sum_{k=1}^{\infty}\frac{H_k(\mu)}{q^{k+1}}.$$
\end{theorem}

\begin{proof}
By the definition of $\P_\mu$, we have
\begin{equation}\label{loc dim 0}
\log\P_\mu([x_1^n])  = \sum_{i\leq n,q\nmid i}\log\mu([x_1^{n}{|_{\Lambda_i(n)}}])
 =   \sum_{k=1}^{\lfloor\log_q n\rfloor}\sum_{\substack{\frac{n}{q^k} <i\leq \frac{n}{q^{k-1}}\\ q\nmid i}} \log\mu([x_1^n{|_{\Lambda_i(n)}}]).
\end{equation}
Recall that $x_1^n{|_{\Lambda_i(n)}}=x_ix_{iq}x_{iq^2}\cdots x_{iq^{\sharp \Lambda_i -1}}$. So
$$
   \mu([x_1^n{|_{\Lambda_i(n)}}])=\mu([x_ix_{iq}x_{iq^2}\cdots x_{iq^{\sharp \Lambda_i -1}}]).
   $$
Let us write $\mu([x_1^n{|_{\Lambda_i(n)}}])$ in the following way
   $$ \mu([x_1^n{|_{\Lambda_i(n)}}])=\mu([x_i])\prod_{j=1}^{\sharp \Lambda_i -1}
     \frac{\mu([x_ix_{iq}x_{iq^2}\cdots x_{iq^j}])}{\mu([x_ix_{iq}x_{iq^2}\cdots x_{iq^{j -1}}])}.
$$
Now we define a suitable  sequence of functions $(F_k)_{k\ge 1}$ on $\Sigma_m$ in order to express $\mu([x_1^n{|_{\Lambda_i(n)}}])$.
 If $k=i$ such that $q\nmid i$, we define $$F_k(x)=F_i(x)=-\log\mu([x_0]).$$ If $k=iq^j$ with $q\nmid i$ and $j\geq 1$, we define $$F_k(x)=F_{iq^j}(x)=-\log\frac{\mu([x_0,x_1,\cdots,x_j])}{\mu([x_0,x_1,\cdots,x_{j-1}])}.$$
Then, we have the following relationship between $F_{k}$ and $\mu$.
$$
-\log\mu([x_{1}^n{|_{\Lambda_{i}}}]) =\sum_{k\in \Lambda_{i}(n)}F_{k}(x_{|_{\Lambda_{i}}}).
$$
 Substituting this expression into  (\ref{loc dim 0}) we obtain
\begin{equation}\label{loc dim 1}
-\log\P_\mu([x_1^n])=\sum_{k=1}^nF_k(x_{|_{\Lambda_{i(k)}}}).
\end{equation}
Now we check that  the sequence $(F_k)_{k\ge 1}$ verifies the hypothesis (\ref{condition2}) of the law of large numbers (Theorem~\ref{LLN}).
Notice that  for any $x\in\Sigma_m$ and any $j\ge 1$, we have
$$|F_{iq^j}(x)|=\left|\log\frac{\mu([x_0,x_1,\cdots,x_j])}{\mu([x_0,x_1,\cdots,x_{j-1}])}\right|\leq \left|\log\mu([x_0,x_1,\cdots,x_j])\right|.$$
This is because $\log \frac{x}{y}\le \log \frac{1}{x}$ when $0\le x\le y\le 1$.
So, for any $i\in\N^*$ with $q\nmid i$ and $j\geq 0$, we have
$$\E_{\mu}\left(F_{iq^j}(x)\right)^2\leq \sum_{x_0,\cdots,x_j\in S}\mu([x_0,x_1,\cdots,x_j])\left(\log\mu([x_0,x_1,\cdots,x_j])\right)^2.$$


Then by Lemma \ref{lemma loc dim} stated below, we obtain
$$\E_{\mu}\left(F_{iq^j}(x)\right)^2=O(j^2)$$
which implies through Cauchy-Schwarz inequality
$$\E_{\mu}\left|F_{iq^{j_1}}(x) F_{iq^{j_2}}(x)\right|= O((j_1+j_2)^2).$$
This quadratic estimate is more than the exponential estimate required by  the hypothesis (\ref{condition2}).
By the law of large numbers,  we have
\begin{equation}\label{4.2.3}
D(\mathbb{P}_\mu, x) =\frac{1}{\log m}\lim_{n\to\infty}\frac{1}{n }\sum_{j=1}^nF_j= \frac{1}{\log m}\lim_{n\to\infty}\frac{1}{n }\sum_{j=1}^n\E_{\mu}F_j \ \ \P_\mu{\rm -a.e.}
\end{equation}
if the limit in the right side hand exists.

This limit does exist. We are going to compute it.  By (2) of Lemma \ref{lem dec}, we have
\begin{equation}\label{4.2.3-1}
\sum_{k=1}^n\E_{\mu}F_k = \sum_{k=1}^{\lfloor\log_q n\rfloor}\sum_{\substack{\frac{n}{q^k} <i\leq \frac{n}{q^{k-1}}\\ q\nmid i}}\sum_{j=0}^{k-1}\E_{\mu}F_{iq^j}.\end{equation}
By the definition of the sequence $(F_k)_{k\ge 1}$, we have
$$
\sum_{j=0}^{k-1}F_{iq^j}(x)= -\log\mu([x_0,\cdots ,x_{k-1}])
$$
which implies immediately
$$
\sum_{j=0}^{k-1}\E_{\mu}F_{iq^j}=-\E_{\mu}\log\mu([x_0,\cdots ,x_{k-1}])=H_k(\mu).
$$
Then substituting this into (\ref{4.2.3-1}) we get
$$\sum_{k=1}^n\E_{\mu}F_k=\sum_{k=1}^{\lfloor\log_q n\rfloor}\sum_{\substack{\frac{n}{q^k} <i\leq \frac{n}{q^{k-1}}\\ q\nmid i}}H_{k}(\mu)=\sum_{k=1}^{\lfloor\log_q n\rfloor}N(n,q,k)H_k(\mu)$$
where $N(n,q,k)$ is the number of $i$'s such that $\frac{n}{q^k} <i\leq \frac{n}{q^{k-1}}$ and $ q\nmid i$. So, by (3) of Lemma \ref{lem dec}, we obtain
$$\lim_{n\to\infty}\frac{1}{n}\sum_{k=1}^n\E_{\mu}F_k=\lim_{n\to\infty}\sum_{k=1}^{\lfloor\log_q n\rfloor}\frac{N(n,q,k)}{n}H_k(\mu) = (q-1)^2\sum_{k=1}^{\infty}\frac{H_k(\mu)}{q^{k+1}}<\infty.$$

\end{proof}


\begin{remark}
Even if the measure $\mu$ itself is not exact dimensional the telescopic measure $\P_\mu$ is. This is because the $\P_\mu$-measure of a cylinder of length $N$ is governed by the measure $\mu$ on short pieces $\Lambda_i(N)$ while the non-exactness of $\mu$ can be seen only on long cylinders. These short pieces are independent.
\end{remark}

\subsection{An elementary inequality}
In the last proof we have used the following elementary estimation.
For $n\ge 1$, let $$
\mathcal{P}_n:=\left\{p=(p_1,\cdots,p_n)\in \R_+^n,\sum_{i=1}^np_i=1\right\}$$
 be the set of probability vectors. We define $L_n: \ \mathcal{P}_n\longrightarrow \R^+$ by $$L_n(p)=\sum_{i=1}^np_i(\log p_i)^2.$$

\begin{lemma}\label{lemma loc dim}
There exists a constant $D>0$ such that $$\max_{p\in \mathcal{P}_n}L_n(p)\leq (\log n)^2+D \log n.$$
\end{lemma}

\begin{proof}
The function $x \mapsto x(\log x)^2$ is bounded on $[0,1]$ and attains its maximal values $4e^{-2}$ at $x=e^{-2}$. 
Hence the inequality holds for $n=2$ with $D = 8e^{-2}$. Now we prove the inequality by induction on $n$.
 Suppose that the inequality holds for $n\leq N$. Let $p\in\mathcal{P}_{N+1}$ be a maximal point of $L_{N+1}$.
  If $p$ is on the boundary of $\mathcal{P}_{N+1}$, then there exists at least one component $p_{i_0}$ of $p$ such that $p_{i_0}=0$. So $$L_{N+1}(p)=\sum_{1\leq i\leq N+1, i\neq i_0}p_i(\log p_i)^2=L_N(p')$$  where
  $p'=(p_1,\cdots,p_{i_0-1},p_{i_0+1},\cdots,p_{N+1})$ is in $\mathcal{P}_N$. In this case, we can conclude by the hypothesis of
   induction.
Now we suppose that $p$ is not on the boundary of $\mathcal{P}_{N+1}$. We use the method of Lagrange multiplier. Differentiating $L_{N+1}(p)$ yields
$$\frac{\partial L_{N+1}}{\partial p_i}(p)=(\log p_i)^2+2\log p_i,\ \ (1\leq i\leq N+1).$$
So we have
\begin{equation}\label{lemma loc dim1}
(\log p_i)^2+2\log p_i=\lambda,\ \ (1\leq i\leq N+1)
\end{equation}
for some real number $\lambda$.
  Let $a,b$ be the two solutions of the equation $$(\log x)^2+2\log x=\lambda.$$
The components of the maximal point $p=(p_1,\cdots ,p_{N+1})$ have two choices: $a$ or $b$. So
\begin{equation}\label{lemma loc dim3}
L_{N+1}(p)=ka(\log a)^2+(N+1-k)b(\log b)^2,
\end{equation}
where $k$ ($0\leq k\leq N+1$) is the number of $a$'s taken by the components of $p$.  Recall that $ka+(N+1-k)b=1$.
Notice that
\begin{equation}\label{lemma loc dim2}
ka(\log a)^2=ka(\log ka-\log k)^2=ka(\log ka)^2+ka (\log k)^2-2ka (\log ka)\log k.
\end{equation}
Since $\max_{x\in [0,1]}-x\log x=\frac{1}{e}$ and $\max_{x\in [0,1]}x(\log x)^2=\frac{4}{e^2}$,
 we get
$$ka(\log a)^2\leq \frac{4}{e^2}+ka (\log k)^2+\frac{2}{e}\log k\le \frac{4}{e^2}+ka (\log (N+1))^2+\frac{2}{e}\log (N+1).$$
A similar estimate holds for $(N+1-k)b(\log b)^2$. Put these two estimates into  (\ref{lemma loc dim3}), we get
$$
P_{N+1}(p) \le \frac{8}{e^2}
+ (\log (N+1))^2 + \frac{4}{e} \log(N+1).$$
We conclude that the inequality holds with $D= \frac{8}{e^2} + \frac{4}{e}$.
\end{proof}

\section{Non-linear transfer equation}\label{section nonlinear equa}
Our study of $A_n\varphi(x)$ will depend upon a class of special telescopic product measures $\mathbb{P}_\mu$
where $\mu$ is a $(\ell-1)$-Markov measure. Our $(\ell-1)$-Markov measures are nothing but  Markov measures
with $S^{\ell}$ as state space.  The transition probability of such a $(\ell-1)$-Markov measure will be determined by the solution of a non-linear
transfer equation. In this section, we will study this non-linear
transfer equation, find its positive solution and construct the $(\ell-1)$-Markov measure and the corresponding telescopic
product measure.

\subsection{Non-linear transfer equation}
 Let $\mathcal{F}(S^{\ell-1}, \mathbb{R}^+)$ denote the cone of functions defined on $S^{\ell-1}$ taking non-negative real values.
 It is identified with a subset in the Euclidean space $\mathbb{R}^{m^{\ell-1}}$.
Let $A: S^{\ell} \to \R^+$ be a given function. We define a non-linear operator $\mathcal{N}: \mathcal{F}(S^{\ell-1}, \mathbb{R}^+) \to \mathcal{F}(S^{\ell-1}, \mathbb{R}^+)$ by
\begin{equation}\label{non-linear transfer operator}
\mathcal{N}y(a_1, a_2, \cdots, a_{\ell-1})
= \left(\sum_{j \in S} A(a_1,a_2, \cdots, a_{\ell-1}, j)
y(a_2, \cdots,a_{\ell-1}, j)\right)^{\frac{1}{q}}.
\end{equation}


We are interested in positive fixed points of the operator $\mathcal{N}$. That means we are interested in $y\in \mathcal{F}(S^{\ell-1}, \mathbb{R}^+)$
 such that $\mathcal{N}y=y$ and $y(a)>0$ for all $a\in S^{\ell-1}$. In general, such fixed points of $\mathcal{N}$ may not exist.
 If $\mathcal{N}$ admits a positive fixed point, then for each $(a_1,\cdots,a_{\ell-1})\in S^{\ell-1}$, there exists at least one $j\in S$ such that $A(a_1,\cdots,a_{\ell-1},j)$  is strictly positive. In fact, this is also a sufficient condition.

\begin{theorem}\label{existence-unicity trans-equ}
Suppose that $A$ is non-negative and that for every $(a_1,\cdots, a_{\ell-1})\in S^{\ell-1}$ there exists at least one $j\in S$ such that $A(a_1,\cdots,\\ a_{\ell-1}, j)>0$. Then $\mathcal{N}$ has a unique positive fixed point.
\end{theorem}

\begin{proof}
We define a partial order on $\mathcal{F}(S^{\ell-1},\R^+)$, denoted by $\leq$, as follows:
$$y_1\leq y_2\ \Leftrightarrow\ y_1(a)\leq y_2(a), \ \forall a\in S^{\ell-1}.$$
It is obvious that $\mathcal{N}$ is increasing with respect to this partial order, i.e., $$y_1\leq y_2\Rightarrow \mathcal{N}(y_1)\leq \mathcal{N}(y_2).$$\\
\indent {\em Uniqueness.}\
We first prove the uniqueness of the positive fixed point by contradiction. Suppose that there are two distinct positive fixed points $y_1$ and $y_2$ for $\mathcal{N}$. Without loss of generality we can suppose that $y_1\nleq y_2$. Let $$\xi=\inf\{\gamma >1,\  y_1\leq \gamma y_2\}.$$
It is clear that $\xi$ is a well defined real number and $y_1\leq \xi y_2$. Since $y_1\nleq y_2$, we must have $\xi>1$.
On the other hand, by the definition of $\mathcal{N}$, the operator $\mathcal{N}$ is homogeneous in the sense that $$\mathcal{N}(cy)=c^{\frac{1}{q}}\mathcal{N}(y),\ \forall y \in \mathcal{F}(S^{\ell-1},\R^+),\ \forall c\in\R^+.$$
It follows that $$y_1=\mathcal{N}(y_1)\leq \mathcal{N}(\xi y_2)=\xi^{\frac{1}{q}}\mathcal{N}( y_2)=\xi^{\frac{1}{q}}y_2.$$
This is a contradiction to the minimality of $\xi$ for $\xi^{\frac{1}{q}}<\xi$.\\
\indent {\em Existence.}\ Now
we  prove the existence. Let
$$\theta_1=\left(\min_{a\in S^{\ell}}A(a)\right)^{\frac{1}{q-1}},\ \ \theta_2=\left(m \max_{a\in S^{\ell}}A(a)\right)^{\frac{1}{q-1}}.$$
Consider the restriction of $\mathcal{N}$ on the compact set $\mathcal{F}(S^{\ell-1},[\theta_1,\theta_2])$ consisting of functions on $S^{\ell-1}$ taking values in $[\theta_1,\theta_2]$. By the definitions of $\theta_1$ and $\theta_2$,  the compact set $\mathcal{F}(S^{\ell-1},[\theta_1,\theta_2])$
is $\mathcal{N}$-invariant, i.e.,  $$\mathcal{N}\left(\mathcal{F}(S^{\ell-1},[\theta_1,\theta_2])\right)\subset \mathcal{F}(S^{\ell-1},[\theta_1,\theta_2]).$$
In fact, let $y \in \mathcal{F}(S^{\ell-1},[\theta_1,\theta_2])$ and let $y_{j_0} = \min_j y_j$. Then $y_{j_0}\ge \theta_1$ and $A(a, j_0) \ge \theta_1^{q-1}$ 
for all $a \in S^{\ell-1}$, so that
$$
\mathcal{N} y (a)\ge (A(a, j_0)y_{j_0})^{1/q} \ge \theta_1.$$
The verification of $\mathcal{N} y (a) \le \theta_2$ is even easier.

Now take any function $y_0$ from the compact set $\mathcal{F}(S^{\ell-1},[\theta_1,\theta_2])$.  By the monotonicity of $\mathcal{N}$, we get 
an increasing sequence
$$y_0\leq \mathcal{N}(y_0)\leq \mathcal{N}^2(y_0)\leq \cdots .$$
Since $\mathcal{F}(S^{\ell-1},[\theta_1,\theta_2])$ is compact, the limit
$g=\lim_{n\to\infty}\mathcal{N}^n(y_0)$ exists. It is a fixed point of $\mathcal{N}$.
\end{proof}

From now on, we concentrate on the following special case:
$$ A(a)=e^{s \varphi(a)},\ \ (a\in S^{\ell})$$
where $s\in \mathbb{R}$ is a parameter.
The corresponding operator will be denoted by  $\mathcal{N}_{s}$.
By Theorem \ref{existence-unicity trans-equ}, there exists a unique positive fixed point for $\mathcal{N}_s$. We denote this fixed point by $\psi_s$.
In the following, we are going to study the analyticity and the convexity of the functions $s \mapsto \psi_s(a)$.

\subsection{Analyticity of $s \mapsto \psi_s(a)$}
\begin{proposition}\label{analyticity}
For every $a\in  S^{\ell-1}$, the function $s\to \psi_s(a)$
is analytic on $\R$.
\end{proposition}
\begin{proof}
We consider the map $G: \R\times \R_+^{*m^{\ell-1}}\to \R^{m^{\ell-1}}$ defined by
$$G(s,(z_a)_{a\in S^{\ell-1}})=\left(G_{b}(s,(z_a)_{a\in S^{\ell-1}})\right)_{b\in S^{\ell-1}},$$
where $$G_{(b_1,\cdots,b_{\ell-1})}(s,(z_a)_{a\in S^{\ell-1}})=z_{(b_1,\cdots,b_{\ell-1})}^q-\sum_{j\in S}e^{s\varphi(b_1,\cdots,b_{\ell-1},j)}z_{(b_2,\cdots,b_{\ell-1},j)}.$$
It is clear that $G$ is analytic. By Theorem \ref{existence-unicity trans-equ}, we have
$$G(s,(\psi_s(a))_{a\in S^{\ell-1}})=0.$$
Moreover the uniqueness in Theorem \ref{existence-unicity trans-equ} implies that for any fixed $s\in \R$, $(\psi_s(a))_{a\in S^{\ell-1}}$ is the unique positive vector satisfying the above equation. For practice, in the following we will write $\underline{\psi_s}=(\psi_s(a))_{a\in S^{\ell-1}}$ and $\underline{z}=(z_a)_{a\in S^{\ell-1}}$.

By the implicit function theorem, if the Jacobian matrix
$$D(s)=\left(\frac{\partial G_{a}}{\partial z_b}(s,\underline{\psi_s})\right)_{(a,b)\in S^{\ell-1}\times S^{\ell-1}}$$ is invertible on a point $s_0\in \R$, then there exist a neighbourhood $(s_0-r_0,s_0+r_0)$ of $s_0$, a neighbourhood $V$ of $\underline{\psi_s}$ in $\R^{m^{\ell-1}}$ and a analytic function $f$ on $(s_0-r_0,s_0+r_0)$ taking values in $V$
such that for any $(t,\underline{z})\subset (s_0-r_0,s_0+r_0)\times V$, we have
$$G(t,\underline{z})=0\ \Leftrightarrow\ f(t)=\underline{z}.$$ Then by the uniqueness of $\psi_s$ for fixed $s$, we have  $\underline{\psi_t}=f(t)
$. 
So the functions $s\to \psi_s(a)$ $(a\in S^{\ell-1})$, which are coordinate functions of $f$, are analytic in $(s_0-r_0,s_0+r_0)$.

We now prove that the matrix $D(s)$ is invertible for any $s\in \R$. To this end, we consider the following matrix
$$\widetilde{D}(s)=\left(\psi_s(b) \frac{\partial G_{a}}{\partial z_b}(s,\underline{\psi_s}) \right)_{(a,b)\in S^{\ell-1}\times S^{\ell-1}},
$$ which is the one obtained by multiplying the $b$-th column of $D(s)$ by $\psi_s(b)$ for each $b\in S^{\ell-1}$.
Then we have the following relation between the determinants of $D(s)$ and $\widetilde{D}(s)$:
$$\det(\widetilde{D}(s))=\left(\prod_{a\in S^{\ell-1}}\psi_s(a)\right)\det(D(s)).$$
So we only need to prove that $\widetilde{D}(s)$ is invertible. We will prove this by showing that $\widetilde{D}(s)$ is strictly diagonal dominating and by
applying the Gershgorin circle theorem (also called Levy-Desplanques Theorem) (see e.g. \cite{Varga}). Recall that  a matrix is said to be strictly diagonal dominating if for every row of the matrix, the modulus
 of the diagonal entry in the row is strictly larger than the sum of the modulus of all the other (non-diagonal) entries in that row.

Let $a=(a_1,\cdots,a_{\ell-1})$ be fixed. The function $G_a(s, \cdot)$ depends only on $z_a$ and $z_b$'s with $b=(a_2,\cdots,a_{\ell-1},j)$. So
$$\frac{\partial G_a}{\partial z_b}(s,\underline{\psi_s})\neq 0$$ only if $b=a$ or $b=(a_2,\cdots,a_{\ell-1},j)$ for some $j\in S$.
It is possible that $a= (a_2,\cdots,a_{\ell-1},j)$ for some $j\in S$ and it is actually the case if and only if $a = (j, j, \cdots, j)$.
To effectively apply the implicit function theorem, we only need to show that for any $a=(a_1,\cdots,a_{\ell-1})$, we have
\begin{equation}\label{analyticity 1}
\left|\psi_s(a)\frac{\partial G_{a}}{\partial z_a}(s,\underline{\psi_s})\right|- \sum_{\substack{j\in S, \\
b=(a_2,\cdots,a_{\ell-1},j)\neq a} }\left|\psi_s(b)\frac{\partial G_{a}}{\partial z_b}(s,\underline{\psi_s})\right|>0.
\end{equation}
In fact, we have
$$
\frac{\partial G_{a}}{\partial z_a}(s,\underline{\psi_s}) = \left\{ \begin{array}{ll}
q\psi_s^{q-1}(a)-e^{s\varphi(a,j)} & \textrm{ if } a=(j,\cdots ,j) \textrm{ for some } j\in S,\\
q\psi_s^{q-1}(a) & \textrm{otherwise.}
\end{array} \right.
$$
and for $b=(a_2,\cdots,a_{\ell-1},j)\neq a$, we have
$$\frac{\partial G_{a}}{\partial z_b}(s,\underline{\psi_s})=e^{s\varphi(a,j)}.$$
Then, substituting the last two expressions into (\ref{analyticity 1}), we obtain that  the member at the left hand side of (\ref{analyticity 1}) is equal to
$$q\psi_s^{q}(a)-\sum_{j\in S}e^{s\varphi(a,j)}\psi_s(a_2,\cdots,a_{\ell-1},j)=(q-1)\psi_s^{q}(a)>0.$$
For the last equality we have used the fact that $\psi_s$ is the solution of $\mathcal{N}_s \psi_s = \phi_s$.
\end{proof}

Our function $\psi_s$ is defined on $S^{\ell-1}$. We extend it on $S^{k}$ for all $1\le k \le \ell-2$ by induction on $k$ as follows
$$\psi_s(a)=\left(\sum_{j\in S}\psi_s(a,j)\right)^{\frac{1}{q}},\quad (\forall  a\in  S^k).$$
It is clear that all these functions $\psi_s$ are  strictly positive for all $s\in \R$.

\begin{corollary}
For any $a\in \bigcup_{1\leq k\leq \ell-1}S^k$, the function $s\to \psi_s(a)$ is analytic on $\R$.
\end{corollary}


\bigskip
\bigskip

\subsection{Convexity of $s \mapsto \psi_s(a)$}
In this subsection, we prove that the functions $s\to \psi_s(a)$ for
$a\in \bigcup_{1\leq k\leq \ell-1}S^k$ and the pressure function
$\pv(s)$ are convex functions on $\R$.

The following
lemma is nothing but the Cauchy-Schwarz inequality. We will use it in this form several times in the proof
of the convexity.

\begin{lemma}\label{lemma convexity 1}
Let $(a_j)_{j=0}^{m-1}$ and $(b_j)_{j=0}^{m-1}$ be two sequences of non-negative real numbers. Then
$$
\left(\sum_{j=0}^{m-1}a_jb_j\right)^2 \le \left(\sum_{j=0}^{m-1}a_j b_j^2\right)\left( \sum_{j=0}^{m-1}a_j\right).$$
\end{lemma}
\begin{proof}
We write $a_jb_j = \sqrt{a_j}b_j \cdot  \sqrt{a_j}$ and then use
the Cauchy-Schwarz inequality.
\end{proof}

Let $\theta_{1}^s=\left(\min_{a\in S^{\ell}}e^{s\varphi(a)}\right)^{\frac{1}{q-1}}$.  In the proof of Theorem \ref{existence-unicity trans-equ}, we have shown that
$$\psi_s=\lim_{n\to \infty}\mathcal{N}^n_s(\underline{\theta_{1}^s}),$$
where $\underline{\theta_{1}^s}$ the function on $S^{\ell-1}$ which is constantly equal to $\theta_{1}^s$.
By the definition of $\mathcal{N}_s$, it is obvious that
$$\mathcal{N}^n_s(\underline{\theta^s_1})=(\theta^s_1)^{\frac{1}{q^n}}\mathcal{N}^n(\underline{1}),$$
where $\underline{1}$ is the  function constantly equal to $1$. However,
for any $s\in \R$, we have
 $\lim_{n\to \infty}(\theta_{1}^s)^{\frac{1}{q^n}}=1$, so that
 $$\psi_s=\lim_{n\to \infty}\mathcal{N}^n_s(\underline{1}).$$
The above convergence is actually uniform for $s$ in any compact set of $\R$.
Let $$\psi_{s, n}=\mathcal{N}^n_s(\underline{1}).$$
In order to prove convexity of the functions
$$
    s \mapsto \psi_s(a), \quad \log \sum_{j \in S} \psi_s(b, j), \quad (a\in S^{\ell-1}, b\in S^{\ell-2})
$$
we have only to show those of $$
    s \mapsto \psi_{s, n}(a), \quad \log \sum_{j \in S} \psi_{s, n}(b,j).
$$
Actually we will make a proof by induction on $n$.

Recall that a function $H$ of class $C^2$ is convex if $H''\ge 0$. A function $H$ of class $C^2$ is log-convex if $\log H$ is convex or equivalently
$H'' H \ge (H')^2$.

First we have the following initiation of the induction.

\begin{lemma}\label{lemma convexity initiation} For any $a\in S^{\ell-1}$, the function $s \mapsto \mathcal{L}_s \underline{1} (a)$
is log-convex.
\end{lemma}
\begin{proof}
The log-convexity of $s \mapsto \mathcal{L}_s \underline{1} (a)$ is equivalent to
 $$
     (\mathcal{L}_s \underline{1} (a))^2 \le (\mathcal{L}_s \underline{1} (a))'' (\mathcal{L}_s \underline{1} (a)).
 $$
Recall the definition of $\mathcal{L}_s \underline{1} (a)$:
$$
   \mathcal{L}_s \underline{1} (a) = \sum_{j\in S} e^{s \varphi(Ta, j)}.
$$
Notice that
$$\left(e^{s\varphi(a,b)}\right)'=e^{s\varphi(a,b)}\varphi(a,b),\quad
 \left(e^{s\varphi(a,b)}\right)''=e^{s\varphi(a,b)}\varphi^2(a,b).$$
 Then log-convexity of $s \mapsto \mathcal{L}_s \underline{1} (a)$ is equivalent to
 $$
    \left(\sum_{j\in S} e^{s \varphi(Ta, j)} \varphi(Ta, j)\right)^2
    \le  \left(\sum_{j\in S} e^{s \varphi(Ta, j)} \varphi(Ta, j)^2\right)
     \left(\sum_{j\in S} e^{s \varphi(Ta, j)} \right).
 $$
 This is nothing  but the Cauchy-Schwarz inequality (see Lemma \ref{lemma convexity 1}).
\end{proof}

The induction will be based on the following recursive relation
$$
                   \psi_{s, n+1}(a) = \mathcal{N}_s \psi_{s, n} (a), \quad \mbox{equivalently}\quad
                   (\psi_{s, n+1}(a))^q = \mathcal{L}_s \psi_{s, n} (a).
$$
We are going to show that if $s\mapsto \mathcal{L}_s\psi_{s, n} (a)$ is log-convex, then so is
$s\mapsto \mathcal{L}_s\psi_{s, n+1} (a)$ and even  $s \mapsto \mathcal{N}_s\psi_{s, n} (a)=\psi_{s, n+1}(a)$
is convex and
$$
      \sum_{j\in S} \psi_{s, n+1}(b, j)
$$
is log-convex.


\begin{lemma}\label{lemma convexity 2}
Let   $(u_s)_{s\in \R}$ be a family of functions  in $\mathcal{F}(S^{\ell-1})$. We suppose that for $a\in S^{\ell-1}$,  $s \mapsto u_s(a)$ is twice differentiable with respect to $s\in \R$. Let
$$
      v_s(a)=\mathcal{N}_{s} u_s (a).
$$
Suppose that for any $a\in S^{\ell-1}$, $s \mapsto \mathcal{L}_s u_s(a)$ is log-convex. Then\\
\indent {\rm (1)}\  For all $a\in S^{\ell-1}$,  $s\mapsto v_s (a)$ is convex.\\
\indent {\rm (2)}\  For all $b\in S^{\ell-2}$,  $s \mapsto \sum_{j \in S}v_s(b, j)$ is log-convex.\\
\indent {\rm (3)}\  For all $a\in S^{\ell-1}$,  $s \mapsto \mathcal{L}_s v_s(a)$ is log-convex.
\end{lemma}

\begin{proof} By the hypothesis, for each $ a\in S^{\ell-1}$, the function $s \mapsto \mathcal{L}_s u_s(a)$ is log-convex. That is to say, if we let $H_s(a) =  \mathcal{L}_s u_s(a)$, we have
\begin{equation}\label{Conv0}
H_s''(a)H_s(a)\geq \left(H'_s(a)\right)^2,
\end{equation}
where, as well as in the following,  $'$ and $''$ will refer to the derivatives with respect to $s$.

(1) Since $v_s (a) = (H_s(a))^{1/q}$, we have
$$(v_s (a))'=\frac{1}{q}(H_s(a))^{\frac{1}{q}-1}H'_s(a).$$
In other words,
\begin{equation}\label{Conv1}
       (v_s (a))'=v_s (a) R_s(a)
\end{equation}
with
$$
   R_s(a) = \frac{1}{q} \frac{H'_s(a)}{ H_s(a)}.
$$
Furthermore we have
\begin{eqnarray*}
(v_s (a))'' & = & \frac{1}{q}\left(\frac{1}{q}-1\right)
(H_s(a))^{\frac{1}{q}-2}
[H'_s(a)]^2+\frac{1}{q}(H_s(a))^{\frac{1}{q}-1} H''_s(a)\\
 & = & \frac{1}{q^2}
(H_s(a))^{\frac{1}{q}-2}
[H'_s(a)]^2 + \frac{1}{q}(H_s(a))^{\frac{1}{q}-2} [H_s(a) H''_s(a)
  - (H'_s(a))^2].
\end{eqnarray*}
By the hypothesis (\ref{Conv0}), $(v_s (a))''\ge 0$. Thus we have proved (1). The last equality implies
$$ (v_s (a))'' \ge \frac{1}{q^2}
(H_s(a))^{\frac{1}{q}-2}
[H'_s(a)]^2.$$
In other words,
\begin{equation}\label{Conv2}
       (v_s (a))''\ge v_s (a) [R_s(a)]^2.
\end{equation}
The relations (\ref{Conv1}) and (\ref{Conv2}) will be useful later.

(2)
By (\ref{Conv2}), we have
$$
\left(\sum_{j \in S} (v_s (b, j))'' \right)
\left(\sum_{j \in S}  v_s (b, j) \right)
\geq \left(\sum_{j \in S} v_s (b, j) R_s(b, j)^2 \right)
\left(\sum_{j \in S}  v_s (b, j) \right).
$$
Then, by the Cauchy-Schwarz inequality in the form of  Lemma \ref{lemma convexity 1}, we have
\begin{eqnarray*}
\left(\sum_{j \in S} (v_s (b, j))'' \right)
\left(\sum_{j \in S}  v_s (b, j) \right)
 & \geq &
\left(\sum_{j \in S} v_s (b, j) R_s(b, j) \right)^2\\
& = &
\left(\sum_{j \in S} (v_s (b, j))' \right)^2
\end{eqnarray*}
where the last equality is due to (\ref{Conv1}).
 Thus we have proved (2).

(3) Recall that
$$
   \mathcal{L} v_s(a) = \sum_{j \in S} e^{s \varphi(Ta, j)} v_s (Ta, j).
$$
Notice that

$\frac{d}{ds} e^{s \varphi(a,b)}v_s(Ta,j)$
$$=
e^{s \varphi(a,j)} \left[\varphi(a,j)v_s(Ta,b)+(v_s(Ta,j))^{'}\right],\ \ \ \ \ \ \ \ \ \ \ \ \ \ \ \ \ \ \ \ \  $$

$\frac{d^2}{ds^2} e^{s \varphi(a,b)}v_s(Ta,j)$
$$\ \ \ \ \ \ \ \ \ \ \ \ = e^{s\varphi(a,j)}\left[\varphi^2(a,j) v_s(Ta,j)+ 2\varphi(a,j) (v_s(Ta,j))'+ (v_s(Ta,j))''\right].
$$
By using (\ref{Conv1}), we can write
$$\frac{d}{ds} e^{s \varphi(a,b)}v_s(Ta,j) =
e^{s\varphi(a,j )}v_s(Ta,j)\left[\varphi(a,j)+R_s(Ta,j)\right].$$
By using (\ref{Conv1}) and (\ref{Conv2}), we get
$$
   \varphi^2(a,j) v_s(Ta,j)+ 2\varphi(a,j) (v_s(Ta,j))'+ (v_s(Ta,j))'' \ge \left[\varphi(a,j) +R_s(Ta,j)\right]^2,
$$
so that
\begin{eqnarray*}
\frac{d^2}{ds^2} e^{s \varphi(a,b)}v_s(Ta,j)  \ge  e^{s\varphi(a,j)} v_s(Ta,j)\left[\varphi(a,j) +R_s(Ta,j)\right]^2.
\end{eqnarray*}
There
$$
(\mathcal{L}_s v_s(a))'' \mathcal{L}_s v_s(a)  \geq \left(\sum_{j\in S} C_s(Ta,j)D_s(a,j)^2\right) \left(\sum_{j\in S}C_s(a,j) \right).
$$
where
$$C_s(a,j)= e^{s\varphi(a,j)} v_s(Ta,j), \quad\ D_s(a,j)= \varphi(a,j) +R_s(Ta,j).$$
Then, by the Cauchy inequality  (see Lemma \ref{lemma convexity 1}), we finally get
$$
(\mathcal{L}_s v_s(a))'' \mathcal{L}_s v_s(a) \geq \left(\sum_{j \in S}C_s(a,j)D_s(a,j)\right)^2
=[(\mathcal{L}_s v_s(a))']^2.
$$
That is the log-convexity of $s \mapsto \mathcal{L}_s v_s(a)$.
\end{proof}

\begin{theorem}\label{thm convexity}
For any $a\in \bigcup_{1\leq j\leq \ell} S^{\ell-j}$, the function $s\mapsto \psi_s(a)$ is convex.
The pressure function $P_\varphi(s)$ is also convex.
\end{theorem}

\begin{proof}  We prove convexity of $s\mapsto \psi_s(a)$ for $a\in  S^{\ell-1}$
by showing those of $s\mapsto \psi_{s, n}(a)$ by induction on $n$. The induction is based on Lemma~\ref{lemma convexity initiation}
and Lemma~\ref{lemma convexity 2} (only the points (1) and (3) are used).

Now we prove convexity of $s\mapsto \psi_s(a)$ for $a\in  S^{\ell-k}$ ($2\le k\le \ell$) by induction on $k$ and by using
what we have just proved above (as the initiation of induction). We can do that because of the following recursive relation:
for $a\in  S^{\ell-k}$ ($2\le k\le \ell$), we have
$$
      \psi_s(a)^q = \sum_{j\in S}\psi_s(a, j).
$$


The right hand side is the operator $\mathcal{L}_s$ defined by the $\varphi$ which is identically zero.
So the log-convexity of $\psi_s(a, j)$ implies that of $\psi_s(a)$ just as the log-convexity of $\psi_{s,n}$ implies that of $\psi_{s,n+1}(a)$.

Recall that the pressure function is proportional to
$$
    s \mapsto \log \psi_s(\emptyset) = \log \sum_{j \in S} \psi_s(j).
$$
The convexity of the pressure is just the log-convexity of $\sum_{j \in S} \psi_s(j)$, which is implied by Lemma~\ref{lemma convexity 2} (3) and the
log-convexity of $\psi_s(j)$.
\end{proof}

\subsection{Construction of the measures $\mu_s$ and $\mathbb{P}_{\mu_s}$}
Below we construct a class of $(\ell-1)$-Markov measure $\mu_s$ whose transition probability and initial law are determined by the fixed point $\psi_s$ of the operator $\mathcal{N}_s$. The corresponding telescopic product measure $\P_{\mu_s}$ will play the same role as Gibbs measure played in the study of simple ergodic averages.

Fix $s\in \R$. Let $\psi_s$ be the function mentioned above.
Recall that $\psi_s$ was first defined on $S^{\ell-1}$ as follows
$$
\left(\psi_s (a)\right)^q
= \sum_{b \in S} e^{s \varphi(a, b)}
\psi_s (Ta, b), \quad (a\in S^{\ell-1}).
$$
Then it was extended  on $S^k$ by induction on $1\le k \le \ell-2$ as follows
$$
\psi_s (a)=\left(\sum_{b \in S} \psi_s (a, b)\right)^{\frac{1}{q}}, \quad \ (a\in S^k).
$$
These functions defined on words of length varying from $1$ to $\ell-1$ allow us to define  a $(\ell-1)$-step Markov measure on $\Sigma_m$, which will
  be denoted by  $\mu_s$,  with the initial law
\begin{equation}\label{def measure 1}
\pi_s([a_1,\cdots,a_{\ell-1}])=\prod_{j=1}^{\ell-1}\frac{\psi_s(a_1,\cdots,a_j)}{\psi_s^q(a_1,\cdots,a_{j-1})}
\end{equation}
and the transition probability
\begin{equation}\label{def measure 2}
Q_s\left([a_1,\cdots,a_{\ell-1}],[a_2,\cdots,a_{\ell}]\right)=e^{s\varphi(a_{1},\cdots,a_{\ell})} \frac{\psi_s(a_2,\cdots,a_{\ell})}{\psi_s^q(a_{1},\cdots,a_{\ell-1})}.
\end{equation}

Here we have identified $\Sigma_m$ with $(S^{\ell-1})^\mathbb{N}$. Actually, $\pi_s$ is a probability vector because
    $$
       \sum_{a_j \in S} \frac{\psi_s(a_1,\cdots,a_j)}{\psi_s^q(a_1,\cdots,a_{j-1})}=1
    $$
    and $Q$ is a transition probability because $\mathcal{N}_s \psi_s = \psi_s$.

As usual, $\mathbb{P}_{\mu_s}$ will denote the telescopic product measure associated to $\mu_s$. See \S 2.1 for its definition
and its general properties.

\section{Properties of the pressure  function}
We have seen in the previous section that the pressure function is real analytic and convex on $\R$. In this section we continue to discuss some of its  further properties. These properties mainly concern its strict convexity when $\alpha_{\min}<\alpha_{\max}$ and a Ruelle type formula relating the expected limit of the multiple ergodic average with respect to the measure $\P_{\mu_s}$ and the derivative of $\pv$.

\subsection{Ruelle type formula}
We state here the following identity which can be regarded as an analogue of Ruelle's derivative formula concerning the classical Gibbs measure and pressure function, its proof will be given in Section \ref{Ruelle formula} (Proposition \ref{prop 1}).
\begin{theorem}\label{Thm Ruelle formula} We have
$$(q-1)^2\sum_{k=1}^\infty\frac{1}{q^{k+1}}\sum_{j=0}^{k-1}\E_{\mu_s}\varphi(x_j,\cdots,x_{j+\ell-1})=P_\varphi'(s).$$
\end{theorem}

As an applications of Theorem \ref{Thm Ruelle formula}, we give the following formula concerning the value $P'_\varphi(0)$.

\begin{proposition}\label{prop average1}
$$P'_\varphi(0)=\frac{\sum_{a\in S^{\ell}}\varphi(a)}{m^{\ell}}.$$
\end{proposition}

\begin{proof}
By Theorem \ref{Thm Ruelle formula}, we have
\begin{equation}\label{equ prop average}
P'_\varphi(0)=(q-1)^2\sum_{k=1}^\infty\frac{1}{q^{k+1}}\sum_{j=0}^{k-1}\E_{\mu_0}
\varphi(x_j,\cdots,x_{j+\ell-1}).
\end{equation}

First of all, we need to determine $\mu_0$. It is straightforward to verify that the constant function $\psi_0\equiv m^{\frac{1}{q-1}}$ is a solution of the following equations when $s=0$.
$$
\left(\psi_s (a)\right)^q
= \sum_{b \in S} e^{s \varphi(a, b)}
\psi_s (Ta, b), \quad (a\in S^{\ell-1}).
$$
Actually, the function $\psi_0$ is the only positive solution by uniqueness of the positive solution (Theorem \ref{existence-unicity trans-equ}). The measure $\mu_0$ defined by this solution as in (\ref{def measure 1}) and (\ref{def measure 2}) is the Lebesgue measure. So, for any  $j\geq 0$ we have

\begin{eqnarray*}
\E_{\mu_0}
\varphi(x_j,\cdots,x_{j+\ell-1}) & = & \sum_{x_0,\cdots,x_{j+\ell-1}}\mu_0([x_0^{j+\ell-1}])\varphi(x_j,\cdots,x_{j+\ell-1}) \\
 & = & \sum_{x_0,\cdots,x_{j+\ell-1}}m^{-(j+\ell)}\varphi(x_j,\cdots,x_{j+\ell-1}) \\
 & = & \sum_{x_0,\cdots,x_{\ell-1}}m^{-\ell}\varphi(x_0,\cdots,x_{\ell-1})\\
 &=& \frac{\sum_{a\in S^{\ell}}\varphi(a)}{m^{\ell}}.
\end{eqnarray*}

Now we get the desired result by substituting the above expression in (\ref{equ prop average}) and by an elementary calculation.
\end{proof}

\subsection{Translation via linearity}

\begin{theorem}\label{thm relation pressure}
For any $\beta\in \R$,  we have $$P_\varphi(s)-\beta
s=P_{\varphi-\beta}(s),$$ where $P_{\varphi-\beta}(s)$ is the
pressure function associated to the potential $\varphi-\beta$.
\end{theorem}

\begin{proof}
Let $\mathcal{N}_{\varphi-\beta, s}$ be the operator as defined in (\ref{non-linear transfer operator})  with
$$ A(a)=e^{s(\varphi(a)-\beta) },\ \ (a\in S^{\ell}).$$
By Theorem \ref{existence-unicity trans-equ}, the operator
$\mathcal{N}_{\varphi-\beta, s}$ admits a unique positive fixed
function $g_s\in \mathcal{F}(S^{\ell-1})$. We have seen that $g_s$
is given by
$$g_s=\lim_{n\to\infty}\mathcal{N}_{\varphi-\beta, s}^n(\underline{1}).$$
By the definitions of $\mathcal{N}_s$ and
$\mathcal{N}_{\varphi-\beta, s}$, it is obvious that
$$\mathcal{N}_{\varphi-\beta, s}=e^{-\frac{s\beta}{q}}\mathcal{N}_s.$$
By induction we get that $$\mathcal{N}_{\varphi-\beta, s}^n=e^{-s\beta(\frac{1}{q}+\cdots+\frac{1}{q^n})}\mathcal{N}_s^n.$$
Thus
$$g_s=\lim_{n\to\infty}\mathcal{N}_{\varphi-\beta, s}^n(\underline{1})=
e^{-s\beta(\sum_{n=1}^\infty\frac{1}{q^n})}\psi_s=e^{-\frac{s\beta}{q-1}}\psi_s.$$
Since for $u\in\bigcup_{1\leq k\leq \ell-2}S^k$, $g_s(u)$ is defined by
$$g_s(u)=\left(\sum_{j=0}^{m-1}g_s(u,j)\right)^{\frac{1}{q}},
$$
we deduce that for $u\in S^k$ with $1\leq k\leq \ell-2$ we have
$$g_s(u)=e^{-\frac{s\beta}{(q-1)q^{\ell-1-k}}}\psi_s(u).$$
Thus
$$P_{\varphi-\beta}(s)=(q-1)q^{\ell-2}\log \sum_{j=0}^{m-1}g_s(j)=-s\beta+P_{\varphi}(s).$$
\end{proof}

\begin{remark}\label{remark} Note that when $\beta =\alpha_{min}$ (resp. $\beta=\alpha_{max}$), the function
$$s\longmapsto \mathcal{N}_{\varphi-\beta, s}$$ is increasing (resp. decreasing).
Then in this case, the function $s\ \mapsto\ g_s $ is also
increasing (resp. decreasing) and so is the pressure function
$s\mapsto P_{\varphi-\beta}(s)$.
\end{remark}

As an application of Theorem \ref{thm relation pressure} and Remark \ref{remark} we have the following consequence.

\begin{proposition}\label{prop cte}
If $s \mapsto P'_\varphi(s)$ is constant on $\R$, then $\varphi$ is
constant on $S^{\ell}$.
\end{proposition}
\begin{proof}
Suppose that $P'_\varphi$ is constant on $\R$. Then
$$P'_\varphi(s)\equiv P'_\varphi(0)=\frac{\sum_{a\in S^{\ell}}\varphi(a)}{m^{\ell}}:=\overline{\varphi}.$$
By Theorem \ref{thm relation pressure}, we have
$$P_\varphi(s)=\overline{\varphi} s+P_{\varphi-\overline{\varphi}}(s).$$
The last two equations imply that
$$P'_{\varphi-\overline{\varphi}}(s)\equiv0.$$
This is equivalent to that
\begin{equation}\label{equ prop cte}
\sum_{j=0}^{m-1}g_s'(j)\equiv0,
\end{equation}
where $g_s$ is the positive fixed point of $\mathcal{N}_{{\varphi-\overline{\varphi}},s}$.
By Theorem \ref{thm convexity}, the function $s\mapsto g_s$ is convex, so $g_s'(j)$ is increasing for all $j\in S$. This, with (\ref{equ prop cte}) imply that $g_s'(j)$ is constant for all $j\in S$. So for every $j$ the function $g_s(j)$ is affine. But these functions are strictly positive on $\R$, they are therefore necessarily constant on $\R$. So $$g_s'(j)\equiv0,\ \ \forall j\in S.$$
For $u\in \bigcup_{1\leq k\leq \ell-2}$, $g_s(u)$ is defined by the following inductive relation.
$$g_s(u)^q=\sum_{j=0}^{m-1}g_s(uj), \ \ u\in \bigcup_{1\leq k\leq \ell-2}S^k.$$
Differentiating  these equations, we get
$$qg_s^{q-1}(u)g_s'(u)=\sum_{j=0}^{m-1}g_s'(uj), \ \ u\in \bigcup_{1\leq k\leq \ell-2}S^k.$$
For any $i\in S$, since $g_s'(i)\equiv0,$ we get
$$\sum_{j=0}^{m-1}g_s'(ij)\equiv0.$$
With the same argument used for proving that $g_s(j)$ is constant for all $j\in S$, we can also prove that
$g_s(ij)$ is constant for all $(i,j)\in S^2$. By induction, we can show that $g_s(u)$ are constant for all $u\in \bigcup_{1\leq k\leq \ell-1}S^{k}$. By the definition of $g_s$, for $u\in S^{\ell-1}$, we have
\begin{equation}\label{equ prop cte1}
g_s^q(u)=\sum_{j=0}^{m-1}e^{s(\varphi(uj)-\overline{\varphi})}g_s(Tu,j).
\end{equation}
We now suppose that $\varphi$ is not constant on $S^{\ell}$, i.e., $\alpha_{\min}<\alpha_{\max}$. Then there exists
$a\in S^{\ell}$ such that $$\varphi(a)>\overline{\varphi}.$$
Let us write $a=(u,j)$ with $u\in S^{\ell-1}$ and $j\in S$.
By  (\ref{equ prop cte1}), we have
$$g_s^q(u)>e^{s(\varphi(u,j)-\overline{\varphi})}g_s(Tu,j), \ \ \forall s\in \R.$$
As $g_s(u)$ and $g_s(Tu,j)$ are strictly positive constants, this is impossible when $s$ tend to $+\infty$. Then we conclude that $\varphi$ is constant on $S^{\ell}$. 
\end{proof}

\subsection{Strict convexity of the pressure function} 
\begin{theorem}\label{thm strict convexity}
Suppose that $\alpha_{\min}<\alpha_{\max}$. Then

 {\rm (i)} $\pv'(s)$ is
strictly increasing on  $\R$.

{\rm (ii)} $\alpha_{\min}\leq
P'_\varphi(-\infty)<P'_\varphi(+\infty)\leq \alpha^f_{\max}.$

\end{theorem}

\begin{proof}

(i) {\em $P'_\varphi(s)$ is strictly increasing on $\R$.} We know
that $P'_\varphi$ is increasing on $\R$ as $P_\varphi$ is convex on
$\R$. Suppose that $P'_\varphi$ is not strictly increasing on $\R$.
Then there exists an interval $[a,b]$ with $a<b$ such that
$P'_\varphi$ is constant on $[a,b]$. On the other hand, we know that
$P_\varphi$ is analytic and so is $P'_\varphi$. Therefore
$P'_\varphi$ must be constant on the whole line $\mathbb{R}$. It is
impossible by Proposition \ref{prop cte} as $\varphi$ is supposed to
be no constant on $S^\ell$.


(ii)\ {\em  $\alpha_{\min }\leq
P'_{\varphi}(-\infty)<P'_{\varphi}(+\infty)\leq \alpha_{\max}$}. The
strict inequality $P'_{\varphi}(-\infty)<P'_{\varphi}(+\infty)$ is
implied by (i). Let us prove the first inequality. The third
inequality can be similarly proved. By  Theorem \ref{thm
relation pressure}, we have
$$P_{\varphi}(s)=\alpha_{\min}s+P_{\varphi-\alpha_{\min}}(s).$$
By Remark \ref{remark}, the function $s\mapsto
P_{\varphi-\alpha_{\min}}(s)$ is increasing. Thus we have
$$P'_{\varphi}(s)=\alpha_{\min}+P'_{\varphi-\alpha_{\min}}(s)\geq \alpha_{\min}$$
which holds for all $s\in \R$.  Letting $s\to -\infty$, we get
$$\alpha_{\min }\leq P'_{\varphi}(-\infty).$$
\end{proof}

To finish this section, we announce the following results concerning the extremal values of $P'_\varphi$ at infinite. Its proof will be given in Section \ref{derivative pressure infinite}.

\begin{theorem}\label{thm extreme1}
We have the equality $$P'_{\varphi}(-\infty)=\alpha_{\min}$$ if and only if there exists an $x=(x_i)_{i=1}^\infty\in \Sigma_m$ such that $$\varphi(x_k,x_{k+1},\cdots,x_{k+\ell-1})=\alpha_{\min},\
\forall k\geq 1.$$ We have an analogue criterion for
$P'_{\varphi}(+\infty)=\alpha_{\max}.$
\end{theorem}

\begin{remark}
 We have a proof of three pages by combinatorially  analyzing
$P_\varphi$. But we would like to give another proof in Section
\ref{derivative pressure infinite} (see Proposition \ref{thm extreme2}),
which is shorter, more intuitive and easier to understand.
\end{remark}

\bigskip

\section{Gibbs property of   $\P_{\mu_s}$}
 In the following we are going to establish a relation between the mass $\P_{\mu_s}([x_1^n])$
and the multiple ergodic sum $\sum_{j=1}^{J} \varphi(x_j\cdots x_{jq^{\ell-1}})$. This can be regarded as the Gibbs property of the measure $\P_{\mu_s}$.

\subsection{Dependence of the Local behavior of $\P_{\mu_s}$  on $\varphi(x_j\cdots x_{jq^{\ell-1}})$.}
There is an explicit relation between the mass $\P_{\mu_s}([x_1^n])$
and the multiple ergodic sum $\sum_{j=1}^{\lfloor\frac{n}{q^{\ell-1}}\rfloor} \varphi(x_j\cdots x_{jq^{\ell-1}})$.
Before stating this relation, we introduce some notation.

Recall that for any integer $k\in \N^*$ we denote by $i(k)$ the unique integer such that
$$k=i(k)q^j, \ \ \ q\nmid i(k).$$
 We associate to $k$ a finite set of integers $\lambda_k$ defined by
 $$\lambda_k\ := \left\{ \begin{array}{lcl}
                           \{i(k),i(k)q,\cdots,i(k)q^j\} & {\rm if } & j<\ell-1 \\
                           \{i(k)q^{j-(\ell-1)},\cdots,i(k)q^j\}  & {\rm if } & j\geq\ell-1.
                         \end{array}\right.
 $$
We define $\lambda_\alpha$ to be the empty set if $\alpha$ is not an integer.
For any sequence $x=(x_i)_{i=1}^\infty \in \Sigma_m$, we denote by $x_{|_{\lambda_k}}$ the restriction of $x$ on $\lambda_k$.

For $x\in \Sigma_m$, we define $$B_n(x)=\sum_{j=1}^n\psi_s(x_{|_{\lambda_j}}).$$
The following basic formula  is a  consequence of the definitions of $\mu_s$ and $\mathbb{P}_{\mu_s}$.

\begin{proposition}\label{prop basic formula loc dim}
We have
$$\log \P_{\mu_s}([x_1^n])=s\sum_{j=1}^{\left\lfloor\frac{n}{q^{\ell-1}}\right\rfloor}\varphi(x_j\cdots x_{jq^{\ell-1}})-(n-\lfloor n/q\rfloor) q \log \psi_s (\emptyset)- qB_{\frac{n}{q}}(x)+ B_n(x).$$
\end{proposition}
\begin{proof}
By the definition of $\P_{\mu_s}$, we have
\begin{equation}\label{unper bounds 1}
\log \P_{\mu_s}([x_1^n])=\sum_{q\nmid i,i\leq n}\log\mu_s([x_1^n|_{\Lambda_i(n)}]).
\end{equation}
However, by the definition of $\mu_s$, if  $\sharp\Lambda_i(n)\leq \ell-1$, we have
\begin{equation}\label{unper bounds 1_1}
\log\mu_s([x_1^n|_{\Lambda_i(n)}])=\sum_{j=0}^{\sharp\Lambda_i(n)-1}\log\frac{\psi_s(x_i,\cdots,x_{iq^j})}{\psi_s^q(x_i,\cdots,x_{iq^{j-1}})}=\sum_{k\in \Lambda_i(n)}\log\frac{\psi_s(x_{|_{\lambda_k}})}{\psi_s^q(x_{|_{\lambda_{k/q}}})}.
\end{equation}

If $\sharp\Lambda_i(n)\geq \ell$, $\log\mu_s([x_1^n|_{\Lambda_i(n)}])$ is equal to
$$ \sum_{j=0}^{\ell-2}\log\frac{\psi_s(x_i,\cdots,x_{iq^j})}{\psi_s^q(x_i,\cdots,x_{iq^{j-1}})}+\sum_{j=\ell-1}^{\sharp\Lambda_i(n)-1}\log\frac{\psi_s(x_{iq^{j-\ell+2}},\cdots,x_{iq^j})e^{s\varphi(x_{iq^{j-\ell+1}},\cdots,x_{iq^j})}}{\psi_s^q(x_{iq^{j-\ell+1}},\cdots,x_{iq^{j-1}})}
$$
$$ =\sum_{j=0}^{\sharp\Lambda_i(n)-1}\log\frac{\psi_s(x_i,\cdots,x_{iq^j})}{\psi_s^q(x_i,\cdots,x_{iq^{j-1}})} + s\sum_{j=\ell-1}^{\sharp\Lambda_i(n)-1}\varphi(x_{iq^{j-\ell+1}},\cdots,x_{iq^j}),
$$
 in other words,
 \begin{equation}\label{unper bounds 1_2}
\mu_s([x_1^n|_{\Lambda_i(n)}])
=  \sum_{k\in \Lambda_i(n)}\log\frac{\psi_s(x_{|_{\lambda_k}})}{\psi_s^q(x_{|_{\lambda_{\frac{k}{q}}}})}+s\sum_{k\in \Lambda_i(n),k\leq n}\varphi(x_{|_{\lambda_k}}).
\end{equation}
Substituting (\ref{unper bounds 1_1}) and (\ref{unper bounds 1_2}) into (\ref{unper bounds 1}), we get
\begin{equation}\label{upper bounds 2}
\log \P_{\mu_s}([x_1^n])= S_n' + s S_n''
\end{equation}
where
$$S_n' =
\sum_{q\nmid i,i\leq n} \sum_{k\in \Lambda_i(n)}\log\frac{\psi_s(x_{|_{\lambda_k}})}{\psi_s^q(x_{|_{\lambda_{\frac{k}{q}}}})}
$$
$$
S_n'' =   \sum_{q\nmid i,i\leq n} \sum_{k\in \Lambda_i(n),k\leq n}\varphi(x_{|_{\lambda_k}}).
$$

For any fixed $i$ with $q\nmid i$, we write
$$
  \sum_{k\in \Lambda_i(n)}\log\frac{\psi_s(x_{|_{\lambda_k}})}{\psi_s^q(x_{|_{\lambda_{\frac{k}{q}}}})}
  = \sum_{k\in \Lambda_i(n)}\log \psi_s(x_{|_{\lambda_k}}) - q \sum_{k\in \Lambda_i(n)}\log \psi_s (x_{|_{\lambda_{\frac{k}{q}}}}).
$$
Recall that if we denote $j_0 = \lfloor \log_q \frac{n}{i}\rfloor$ the largest integer such that $iq^{j_0}\le n$, then
$$
    \Lambda_i(n)=\{i, i q, iq^2, \cdots, iq^{j_0}\}.$$
If $k= i$, we have $x_{k/q}=\emptyset$. If $k =iq^j$ with $1\le j\le j_0$, we have $k/q = iq^{j-1}$ which belongs to $\Lambda_i(n)$.
In the following we formally write
$$
  \Lambda_i(n/q)=\{i, i q, iq^2, \cdots, iq^{j_0-1}\}.
$$
Then we can write
$$\sum_{k\in \Lambda_i(n)}\log\frac{\psi_s(x_{|_{\lambda_k}})}{\psi_s^q(x_{|_{\lambda_{\frac{k}{q}}}})}=(1-q)\sum_{k\in \Lambda_i(\frac{n}{q})}\psi_s(x_{|_{\lambda_k}}) - q \log \psi_s (\emptyset) + \sum_{k\in \Lambda_i(n),kq>n}\psi_s(x_{|_{\lambda_k}}).$$
Notice that there  is only one term in the last sum, which corresponds to $k = iq^{j_0}$.  Now we take sum over $i$ to get
\begin{eqnarray*}
S_n'
= (1-q)\sum_{k\leq \frac{n}{q}}\psi_s(x_{|_{\lambda_k}}) - q(n-\lfloor n/q\rfloor)\log \psi_s (\emptyset)+  \sum_{k> \frac{n}{q}}\psi_s(x_{|_{\lambda_k}}),
\end{eqnarray*}
because $\sharp\{i\leq n,q\nmid i\} = n-\lfloor n/q\rfloor$ and
\begin{eqnarray*}
\sum_{i\leq n,q\nmid i}\sum_{k\in \Lambda_i(\frac{n}{q})}\psi_s(x_{|_{\lambda_k}})
= \sum_{k\leq \frac{n}{q}}\psi_s(x_{|_{\lambda_k}}),
\end{eqnarray*}
\begin{eqnarray*}
\sum_{i\leq n,q\nmid i}\sum_{k\in \Lambda_i(n),kq>n}\psi_s(x_{|_{\lambda_k}})
=\sum_{k> \frac{n}{q}}\psi_s(x_{|_{\lambda_k}}).
\end{eqnarray*}
Recall that $B_n(x)=\sum_{j=1}^n\psi_s(x_{|_{\lambda_j}}).$
We can rewrite
\begin{eqnarray*}
    (1-q)\sum_{k\leq \frac{n}{q}}\psi_s(x_{|_{\lambda_k}}) + \sum_{k> \frac{n}{q}}\psi_s(x_{|_{\lambda_k}})
  &= &   -q\sum_{k\leq \frac{n}{q}}\psi_s(x_{|_{\lambda_k}})  +  \sum_{k\le n  }\psi_s(x_{|_{\lambda_k}})
 \\
 & = &- qB_{\frac{n}{q}}(x) + B_n(x).
\end{eqnarray*}
Thus
\begin{eqnarray*}
 S_n ' = -q(n-\lfloor n/q\rfloor)\log \psi_s(\emptyset)- qB_{\frac{n}{q}}(x) + B_n(x).
\end{eqnarray*}
On the other hand, we have
$$S_n''= \sum_{q\nmid i,i\leq n}\ \ \sum_{k\in \Lambda_i(n),k\leq n}\varphi(x_{|_{\lambda_k}})=\sum_{k\leq n}\varphi(x_{|_{\lambda_k}})=\sum_{j=1}^{\left\lfloor\frac{n}{q^{\ell-1}}\right\rfloor}\varphi(x_j\cdots x_{jq^{\ell-1}}). $$
Substituting these expressions of $S_n'$ and $S_n''$  into  (\ref{upper bounds 2}), we get the desired result.
\end{proof}

\bigskip

 \section{Proof of theorem \ref{thm principal}: computation of $\dim_H E(\alpha)$ }\label{Proof thm principal}

We will use the measure $\P_{\mu_s}$ to estimate the dimensions of
levels sets $E(\alpha)$. Actually, for a given $\alpha$, there is
some $s$ such that $\P_{\mu_s}$ is a nice Frostman type measure
sitting on $E(\alpha)$.
  First of all, let us calculate the
local dimensions of $\P_{\mu_s}$.

\subsection{Upper bounds of local dimensions of $\P_{\mu_s}$ on level sets}

We define
$$E^+(\alpha):=\left\{x\in \Sigma_m :
\limsup_{n\to\infty}\frac{1}{n}\sum_{k=1}^{n}\varphi(x_k,x_{kq},\cdots,x_{kq^{\ell-1}})\leq\alpha\right\},$$
and
$$E^-(\alpha):=\left\{x\in \Sigma_m :
\liminf_{n\to\infty}\frac{1}{n}\sum_{k=1}^{n}\varphi(x_k,x_{kq},\cdots,x_{kq^{\ell-1}})
\geq\alpha\right\}.$$ It is clear that
$$
     E(\alpha) = E^+(\alpha) \cap E^-(\alpha).
$$

 In this subsection we
will obtain  upper bounds for local dimensions of $\P_{\mu_s}$ on
the sets $E^+ (\alpha)$ and $E^-(\alpha)$. The following elementary
result will be useful for the estimation of local dimensions
of $\P_{\mu_s}$.

\begin{lemma}\label{upperbounds lemma}
Let $(a_n)_{n\geq 1}$ be a bounded sequence of non-negative real numbers. Then
\[\liminf_{n\to\infty}\left(a_{\lfloor n/q \rfloor}-a_n\right)\leq 0.\]
\end{lemma}

\begin{proof}
Let $b_l=a_{q^{l-1}}-a_{q^{l}}=a_{\frac{q^l}{q}}-a_{q^{l}}$ for
$l\in\N^*$. Then the boundedness implies
$$\lim_{l \to \infty }\frac{b_1+\cdots+b_l}{l}= \lim_{l \to \infty } \frac{a_1-a_{q^l}}{l}=0.$$
This in turn implies  $\liminf_{l\to \infty }b_l\leq 0$ so that
$$\liminf_{l\to \infty }\left(a_{\lfloor n/q \rfloor}-a_n\right)\leq \liminf_{l\to \infty }b_l\leq 0.$$
\end{proof}

\begin{proposition}\label{prop loc dim upper bound}
For every $x\in E^+(\alpha)$, we have
$$\forall s\leq0,  \ \ \ \underline{D}(\P_{\mu_s},x)\leq
\frac{P(s) -\alpha s}{q^{\ell-1}\log m}.$$
For every $x\in E^-(\alpha)$, we have
$$\forall s\geq0, \ \ \ \underline{D}(\P_{\mu_s},x)\leq
\frac{P(s) -\alpha s}{q^{\ell-1}\log m}.$$
Consequently, for every $x\in E(\alpha)$, we have
$$\forall s\in \R, \ \ \ \underline{D}(\P_{\mu_s},x)\leq
\frac{P(s) -\alpha s}{q^{\ell-1}\log m}.$$
\end{proposition}
\begin{proof} The proof is based  on
Proposition \ref{prop basic formula loc dim}, which implies that for any $x\in \Sigma_m$ and any $n\geq 1$ we have
 {\setlength\arraycolsep{2pt}
 \begin{eqnarray*}
 -\frac{\log\P_{\mu_s}([x_1^n])}{n}& = & -\frac{s}{n}\sum_{j=1}^{\lfloor\frac{n}{q^{\ell-1}}\rfloor}\varphi(x_j\cdots x_{jq^{\ell-1}})+
  q \frac{n-\lfloor n/q\rfloor}{n}\log \psi_s(\emptyset)
 \nonumber\\
 & & +\frac{B_{\frac{n}{q}}(x)}{\frac{n}{q}} -\frac{B_n(x)}{n}.
 \end{eqnarray*}}
Since  the function $\psi_s$ is bounded, so is the sequence $(B_n(x)/n)_n$. Then,  by Lemma \ref{upperbounds lemma}, we have
$$\liminf_{n\to \infty}\frac{B_{\frac{n}{q}}(x)}{\frac{n}{q}} -\frac{B_n(x)}{n}\leq 0.$$
Therefore
\begin{equation*}\label{UpperEstimateOfD}
\underline{D}(\P_{\mu_s},x)\le \liminf_{n\to \infty} -\frac{s}{n
\log
m}\sum_{j=1}^{\lfloor\frac{n}{q^{\ell-1}}\rfloor}\varphi(x_j\cdots
x_{jq^{\ell-1}})+
 (q -1) \log_m \psi_s(\emptyset).
 \end{equation*}

Now suppose that $x\in E^+(\alpha)$ and $s\leq 0$. Since
$$\liminf_{n\to \infty}\frac{1}{n}\sum_{j=1}^{\lfloor\frac{n}{q^{\ell-1}}\rfloor}\varphi(x_j\cdots x_{jq^{\ell-1}})\leq  \limsup_{n\to \infty}\frac{1}{n}\sum_{j=1}^{\lfloor\frac{n}{q^{\ell-1}}\rfloor}\varphi(x_j\cdots x_{jq^{\ell-1}})\leq\frac{\alpha}{q^{\ell-1}},$$
we have
\begin{eqnarray*}
\liminf_{n\to \infty}-\frac{s}{n}\sum_{j=1}^{\lfloor\frac{n}{q^{\ell-1}}\rfloor}\varphi(x_j\cdots x_{jq^{\ell-1}}) &= & -s\liminf_{n\to \infty}\frac{1}{n}\sum_{j=1}^{\lfloor\frac{n}{q^{\ell-1}}\rfloor}\varphi(x_j\cdots x_{jq^{\ell-1}}) \\
 & \leq& \frac{-s\alpha}{q^{\ell-1}},
\end{eqnarray*}
so that
$$\underline{D}(\P_{\mu_s},x) \leq -\frac{\alpha s}{q^{\ell-1}\log m }+(q-1)\log_m \psi_s(\emptyset)=\frac{P(s) - \alpha s }{q^{\ell-1}\log m}, $$
where the last equation is due to
$$P(s)=(q-1)q^{\ell -2}\log\sum_{j\in S^\ell}\psi_s(j)=(q-1)q^{\ell -2}q\log\psi_s(\emptyset). $$

By an analogue argument, we can prove the same result for $x\in E^-(\alpha)$ and $s\geq 0$.
\end{proof}

\subsection{Range of $L_\varphi$}
Recall that $L_\varphi$ is the set of $\alpha$ such that $E(\alpha)\neq \emptyset$.

\begin{proposition} \label{range}
We have
$L_\varphi\subset [P'_\varphi(-\infty),P'_\varphi(+\infty)].$
\end{proposition}

\begin{proof} We prove it by contradiction.
Suppose that $E(\alpha)\neq \emptyset$ for some
$\alpha<P'_\varphi(-\infty)$. Let $x=(x_i)_{i=1}^\infty\in
E(\alpha)$. Then by Proposition \ref{prop loc dim upper bound}, we
have
\begin{equation}\label{prop borne inf de mesure1}
\liminf_{n\to\infty}\frac{-\log_m\P_{\mu_s}([x_1^n])}{n}\leq
\frac{\pv(s)-\alpha s}{q^{\ell-1}\log m}, \  \forall s\in \R.
\end{equation}
On the other hand, by the mean value theorem, we have
\begin{equation}\label{prop borne inferieure de mes2}
\pv(s)-\alpha s=\pv(s)-\pv(0)-\alpha s+\pv(0)=\pv'(\eta_s)s-\alpha s+\pv(0)
\end{equation} for some real number $\eta_s$ between $0$ and $s$. As  $\pv $ is convex, $\pv'$ is increasing on $\R$. If we assume $s<0$, then we have
$$
\pv'(\eta_s) s -\alpha s+\pv(0) \leq \pv'(-\infty) s-\alpha
s+\pv(0)=\left(\pv'(-\infty)-\alpha \right) s +\pv(0).$$ As
$\pv'(-\infty)-\alpha>0$, we deduce from (\ref{prop borne inferieure
de mes2}) that for $s$ small enough (close to $-\infty$), we have
$\pv(s)-\alpha s<0$. Then by (\ref{prop borne inf de mesure1}),  for
$s$ small  enough we obtain
$$\liminf_{n\to\infty}\frac{-\log_m\P_{{\mu_s}}([x_1^n])}{n}<0$$
which implies $\P_{{\mu_s}}([x_1^n])>1$ for  an infinite number of
$n$'s. This is a contradiction to the fact that $\P_{\mu}$ is a
probability measure on $\Sigma_m$. Thus we have prove that for
$\alpha$ such that $E(\alpha)\not=\emptyset$, we have $\alpha \ge
P'(-\infty)$. Similarly we can also prove $\alpha \le P'(+\infty)$.
\end{proof}

As we shall show, we will have the equality $L_\varphi=
[P'_\varphi(-\infty),P'_\varphi(+\infty)]$.

\subsection{Upper bounds of Hausdorff dimensions of level sets}
A upper bound of the Hausdorff dimensions of levels set is a direct consequence of the Billingsley lemma
and of Proposition \ref{prop loc dim upper bound}.
The Billingsley lemma is stated as follows.

\begin{lemma}[see Prop.4.9 in \cite{Fal90}]\label{Billingsley}
Let $E$ be a Borel set in $\Sigma_m$ and let $\nu$ be a finite Borel measure on $\Sigma_m$.\\
\indent \ {\rm (i)}  We have $\dim_H(E)\geq d$ if $\nu(E) > 0$ and $\underline{D}(\nu,x)\geq d$ for $\nu$-a.e $x$.\\
\indent {\rm (ii)}  We have $\dim_H(E) \leq d$ if $\underline{D}(\nu,x)\leq d$ for all $x \in E$.
\end{lemma}

Recall that
$$
   P_{\varphi}^*(\alpha) = \inf_{s\in \mathbb{R}} (P_\varphi(s)-\alpha s).
$$

\begin{proposition}\label{prop upper bound}
For any $\alpha\in (\pv'(-\infty),\pv'(0))$, we have
$$
\dim_HE^+(\alpha) \leq \inf_{s\leq 0}\frac{1}{q^{\ell-1}\log
m}[-\alpha s+P_\varphi(s)] 
$$
For any $\alpha\in (\pv'(0),\pv'(+\infty))$, we have
$$
\dim_HE^-(\alpha)\leq \inf_{s\geq 0}\frac{1}{q^{\ell-1}\log
m}[-\alpha s+P_\varphi(s)]
$$
In particular, we have
$$\dim_H
E(\alpha)\leq\frac{P_{\varphi}^*(\alpha)}{q^{\ell-1}\log m}.
$$
\end{proposition}

\bigskip
\bigskip

\subsection{Ruelle type formula} \label{Ruelle formula}
This subsection is mainly devoted  to proving  the following identity which was announced in Theorem \ref{Thm Ruelle formula}.
$$(q-1)^2\sum_{k=1}^\infty\frac{1}{q^{k+1}}\sum_{j=0}^{k-1}\E_{\mu_s}\varphi(x_j,\cdots,x_{j+\ell-1})=P_\varphi'(s).$$
This formula will be
useful for estimating the lower bounds of $\dim_H E(\alpha)$.

We need to do some preparations for proving this result.
First of all, we deduce
 some identities concerning the functions
$\psi_s$.

Recall that $\psi_s(a)$ are defined for $a\in \bigcup_{1\leq k\leq
\ell-1}S^k $. They verify the following equations. For $a\in
S^{\ell-1}$, we have
$$\psi_s^q(a)=\sum_{b\in S}e^{s\varphi(a,b)}\psi_s(Ta,b)$$ and
for $a\in S^k$ ($1\le k \le \ell-2$) we have
$$\psi_s^q(a)=\sum_{b\in S}\psi_s(a,b).$$
Differentiating the two sides of each of the above two equations with respect to $s$, we get for all $a\in S^{\ell-1}$
$$q\psi_s^{q-1}(a)\psi_s'(a)=\sum_{b\in S}e^{s\varphi(a,b)}
\varphi(a,b)\psi_s(Ta,b)+\sum_{b\in
S}e^{s\varphi(a,b)}\psi_s'(Ta,b)$$
and for all $a\in
\bigcup_{1\leq k\leq \ell-2}S^k$
$$q\psi_s^{q-1}(a)\psi_s'(a)=\sum_{b\in S}\psi_s'(a,b).$$ Dividing these equations by
$\psi_s^q(a)$ (for different $a$ respectively), we get

\begin{lemma}
For any $a\in S^{\ell-1}$, we have
\begin{equation}\label{identity1}
q\frac{\psi_s'(a)}{\psi_s(a)}=\sum_{b\in
S}\frac{e^{s\varphi(a,b)}\varphi(a,b)\psi_s(Ta,b)}
{\psi_s^{q}(a)}+\sum_{b\in
S}\frac{e^{s\varphi(a,b)}\psi_s'(Ta,b)}{\psi_s^{q}(a)},
\end{equation}
and for any $a\in \bigcup_{1\leq
k\leq \ell-2}S^k$
\begin{equation}\label{identity2}
q\frac{\psi_s'(a)}{\psi_s(a)}=\sum_{b\in
S}\frac{\psi_s'(a,b)}{\psi_s(a,b)}.
\end{equation}
\end{lemma}

We denote
$$w(a)=\frac{\psi_s'(a)}{\psi_s(a)}, \ \ \ \
v(a)=\sum_{b\in
S}\frac{e^{s\varphi(a,b)}\psi_s'(Ta,b)}{\psi_s^{q}(a)}, (\forall
a\in S^{\ell-1}).
$$
 Then we have the following identities.

\begin{lemma}\label{identity} For any $n\in\N$, we have
\begin{eqnarray}\label{E1}
\E_{\mu_s}\varphi(x_n^{n+\ell-1})&=& q\E_{\mu_s}
   w(x_n^{n+\ell-2})-\E_{\mu_s} v(x_n^{n+\ell-2}), \ \ (\forall n \ge 0). \\
\label{E2}
     \E_{\mu_s} w(x_n^{n+\ell-2})&=&\E_{\mu_s} v(x_{n-1}^{n+\ell-3}),\ \ (\forall n \ge 1). \\
\label{E3}
           \E_{\mu_s} w(x_0^{\ell-2})&=&\frac{1}{q(q-1)}\pv'(s).
\end{eqnarray}
\end{lemma}

\begin{proof}
The Markov property of  $\mu_s$ can be stated as follows (see\eqref{def measure 2})
$$
    \mu_s ([x_0^{n+\ell-1}]) = \mu_s ([x_0^{n+\ell-2}]) Q_s(x_n^{n+\ell -1})
$$
where $$ Q_s(x_n^{n+\ell -1}) =
  \frac{e^{s\varphi(x_n^{n+\ell-1})}\psi_s(x_{n+1}^{n+\ell-1})}{\psi_s^q(x_{n}^{n+\ell-2})}.
$$
 By the Markov property, we have
 \begin{eqnarray*}
 \mathbb{E}_{\mu_s} \varphi(x_n^{n+\ell-1})
 & = & \sum_{x_0,\cdots,x_{n+\ell-1}}\mu_s([x_0^{n+\ell-1}])\varphi(x_n^{n+\ell-1})\\
 & = &
 \sum_{x_0,\cdots,x_{n+\ell-2}}\mu_s([x_0^{n+\ell-2}])\sum_{x_{n+\ell-1}}
 Q_s(x_n^{n+\ell -1})
 \varphi(x_n^{n+\ell-1}).
 \end{eqnarray*}
 However, by the definition of $Q_s$ and using (\ref{identity1}), it is straightforward to check that
 $$\sum_{x_{n+\ell-1}}
 Q_s(x_n^{n+\ell -1})
 \varphi(x_n^{n+\ell-1})=qw(x_n^{n+\ell-2})-v(x_n^{n+\ell-2}).$$
 So (\ref{E1}) is  a combination of the above two equations.

To obtain (\ref{E2}), we still use  the Markov property of $\mu_s$,
to get

$$\mathbb{E}_{\mu_s} w(x_n^{n+\ell -2})  =
\sum_{x_0,\cdots,x_{n+\ell-2}}\mu_s([x_0^{n+\ell-2}])w(x_n^{n+\ell-2}) \ \ \ \  \ \ \ \  \ \ \ \  \ \ \ \  \ \ $$
\begin{eqnarray*}
 \ \ \ \  & = & \sum_{x_0,\cdots,x_{n+\ell-3}}\mu_s([x_0^{n+\ell-3}])\sum_{x_{n+\ell-2}}
\frac{e^{s\varphi(x_{n-1}^{n+\ell-2})}\psi_s(x_{n}^{n+\ell-2})}
{\psi_s^q(x_{n-1}^{n+\ell-3})}\frac{\psi_s'(x_n^{n+\ell-2})}{\psi_s(x_n^{n+\ell-2})} \\
 \ \ \ \  & = & \sum_{x_0,\cdots,x_{n+\ell-3}}\mu_s([x_0^{n+\ell-3}]) v(x_{n-1}^{n+\ell-3})
 = \mathbb{E}_{\mu_s} v(x_{n-1}^{n+\ell -3}) .
\end{eqnarray*}

 Now let us treat (\ref{E3}). First of all, by the definition of $w$ and $\mu_s$ we get


$$\psi'_s(x_0^{\ell-3})=\sum_{x_{\ell -2}}\psi'_s(x_0^{\ell -2}) ,$$ hence
\begin{eqnarray*}
 \E_{\mu_s} w(x_0^{\ell-2}) & = & \sum_{x_0,\cdots, x_{\ell-2}}\mu_s([x_0^{\ell-2}])w(x_0^{\ell-2})\\
 & = & \sum_{x_0,\cdots, x_{\ell-3}}\mu_s([x_0^{\ell-3}])\sum_{x_{\ell-2}}\frac{\psi_s'(x_0^{\ell-2})}{\psi_s(x_0^{\ell-3})}.
\end{eqnarray*}
 By (\ref{identity2}), the last sum is equal to
 $q\frac{\psi_s'(x_0^{\ell-3})}{\psi_s(x_0^{\ell-3})}$.
 So
 $$
 \E_{\mu_s} w(x_0^{\ell-2})
 =q\sum_{x_0,\cdots, x_{\ell-3}}\mu_s([x_0^{\ell-3}])\frac{\psi_s'(x_0^{\ell-3})}{\psi_s(x_0^{\ell-3})}.
 $$
 Repeating the same argument,  we obtain by induction on $j$ that
 $$\E_{\mu_s} w(x_0^{\ell-2})
 =q^{\ell-2-j}\sum_{x_0,\cdots, x_{j}}\mu_s([x_0^{j}])\frac{\psi_s'(x_0^{j})}{\psi_s(x_0^{j})}.$$
 So finally when $j=0$ we get
$$
\E_{\mu_s} w(x_0^{\ell-2})=q^{\ell-2}\sum_{b\in
S}\mu_s([b])\frac{\psi_s'(b)}{\psi_s(b)} =q^{\ell-2}\frac{\sum_{b\in
S}\psi_s'(b)}{\sum_{b\in S}\psi_s(b)}=\frac{1}{q(q-1)}\pv'(s)
$$
where we used the fact that
$$
   \mu_s ([b]) = \frac{\psi_s(b)}{\sum_{b\in S} \psi_s(b)}.
$$

\end{proof}

Now, we can prove the Ruelle type formula which was announced in Theorem \ref{Thm Ruelle formula}.
We restate it as the following proposition.

\begin{proposition}\label{prop 1}
For any $s\in\R$, we have
$$(q-1)^2\sum_{k=1}^\infty\frac{1}{q^{k+1}}\sum_{j=0}^{k-1}\E_{\mu_s}
\varphi(x_j,\cdots,x_{j+\ell-1})=P'_\varphi(s).$$
\end{proposition}


\begin{proof}
By (\ref{E1}) in Lemma \ref{identity}, for any $k\in\N^*$, we have
{\setlength\arraycolsep{2pt}
\begin{eqnarray*}
\sum_{j=0}^{k-1}\E_{\mu_s}\varphi(x_j,\cdots,x_{j+\ell-1}) &=&\sum_{j=0}^{k-1}\left(q\E_{\mu_s}w(x_j^{j+\ell-2})-\E_{\mu_s}v(x_j^{j+\ell-2})\right)\\
 & = & q\E_{\mu_s}w(x_0^{\ell-2})+q\sum_{j=1}^{k-1}\E_{\mu_s}w(x_j^{j+\ell-2})
\nonumber\\
& & -\sum_{j=0}^{k-1}\E_{\mu_s}v(x_j^{j+\ell-2}).
\end{eqnarray*}}

Let
$$S_k=\sum_{j=0}^{k-1}\E_{\mu_s}v(x_j^{j+\ell-2}).$$
Then by~\eqref{E2} in Lemma \ref{identity}, we have
$$\sum_{j=1}^{k-1}\E_{\mu_s}w(x_j^{j+\ell-2})=S_{k-1}.$$
Using the above equality and (\ref{E3}) in Lemma \ref{identity}, we
can write
$$\sum_{j=0}^{k-1}\E_{\mu_s}\varphi(x_j,\cdots,x_{j+\ell-1})=\frac{\pv'(s)}{q-1}+qS_{k-1}-S_k.$$
The facts $S_0=0$ and $S_k =o(k)$ imply $$
    \sum_{k=1}^\infty \frac{1}{q^{k+1}} (q S_{k-1} - S_k) =0.
 $$
 Then
\begin{eqnarray*}
(q-1)^2\sum_{k=1}^\infty\frac{1}{q^{k+1}}\sum_{j=0}^{k-1}\E_{\mu_s}
\varphi(x_j,\cdots,x_{j+\ell-1})
=(q-1)^2\sum_{k=1}^\infty\frac{1}{q^{k+1}}\frac{\pv'(s)}{q-1}.
\end{eqnarray*}
which is equal to $\pv'(s)$, because $\sum_{k=1}^\infty 1/q^{k+1}=1/(q-1)$.
\end{proof}

\medskip
\medskip

\subsection{When $P'_\varphi(-\infty)=\alpha_{\min}$ and when $P'_\varphi(+\infty)=\alpha_{\max}$}\label{derivative pressure infinite}

We now give the proof of the statement announced in Theorem \ref{thm extreme1} concerning the extremal values of $P'_\varphi$ at infinity.

\begin{theorem}\label{thm extreme2}
We have the equality $$P'_{\varphi}(-\infty)=\alpha_{\min}$$ if and only if there exist an $x=(x_i)_{i=0}^\infty\in \Sigma_m$ such that $$\varphi(x_k,x_{k+1},\cdots,x_{k+\ell-1})=\alpha_{\min},\
\forall k\geq 0.$$ We have an analogue criterion for
$P'_{\varphi}(+\infty)=\alpha_{\max}.$
\end{theorem}
\begin{proof}
We give the proof of the criterion for $P'_{\varphi}(-\infty)=\alpha_{\min}$, the one for $P'_{\varphi}(+\infty)=\alpha_{\max}$ is similar.

(1). {\em Sufficient condition.} Suppose that there exists a $(z_j)_{j=0}^\infty\in \Sigma_m$ such that
$$\varphi(z_j,\cdots,z_{j+\ell-1})=\alpha_{\min},\ \ \forall j\geq 0.$$
We are going to prove that $P'_\varphi(-\infty)=\alpha_{\min}$.  By Theorem \ref{thm strict convexity} (ii), we have $P'_\varphi(-\infty)\geq\alpha_{\min}$, thus we only need to show that $P'_\varphi(-\infty)\leq\alpha_{\min}$. Actually we only need to find a $(x_{j})_{j=1}^\infty\in \Sigma_m$ such that
$$\lim_{n\to \infty}\frac{1}{n}\sum_{j=1}^n\varphi(x_j,\cdots,x_{jq^{\ell-1}})=\alpha_{\min},$$
then by Proposition \ref{range}, $\alpha_{\min}\in [P'_\varphi(-\infty),P'_\varphi(+\infty)]$, so $P'_\varphi(-\infty)\leq\alpha_{\min}$. We can do this by choosing the sequence $(x_j)_{j=1}^\infty=\prod_{i\geq 1,\ q\nmid i}(x_{iq^j})_{j=0}^\infty$ with
$$(x_{iq^j})_{j=0}^\infty\ =\ (z_j)_{j=0}^\infty.$$

(2). {\em Necessary condition.} Suppose that there is no $(x_j)_{j=0}^\infty\in \Sigma_m$ such that
$$\varphi(x_j,\cdots,x_{j+\ell-1})=\alpha_{\min},\ \ \forall j\geq 0.$$
We are going to show that there exists an $\epsilon >0$ such that
$$P'_\varphi(s)\geq \alpha_{\min}+\epsilon,\ \ \forall s\in \R.$$ And this will imply that $P'_\varphi(-\infty)\geq \alpha_{\min}+\epsilon$.

From the hypothesis, we deduce that there exist no words $x_0^{n+\ell-1}$ with $n\geq m^\ell $ such that
\begin{equation}\label{equ prop extremal}
\varphi(x_j,\cdots,x_{j+\ell-1})=\alpha_{\min},\ \ \forall 0\leq j\leq n.
\end{equation}
Indeed, as $x_j^{j+\ell-1}\in S^{\ell}$ for all $0\leq j\leq n$ there are at most $m^\ell$ choices for $x_j^{j+\ell-1}$. So for any word $x_0^{n+\ell-1}$ with $n\geq m^\ell $, there exist at least two $j_1,j_2\in \{0,\cdots,n\}$ such that $$x_{j_1}^{j_1+\ell-1}=x_{j_2}^{j_2+\ell-1}.$$
Then if the word $x_0^{n+\ell-1}$ satisfies (\ref{equ prop extremal}), the infinite sequence
$$(y_j)_{j=0}^\infty\ =\ (x_{j_1},\cdots,x_{j_2-1})^\infty$$ would verify that
$$\varphi(y_j,\cdots,y_{j+\ell-1})=\alpha_{\min},\ \ \forall j\geq 0.$$
This is a contradicts the hypothesis. We conclude then that for any word $x_0^{m^\ell+\ell-1}\in S^{m^{\ell}+\ell-1}$ there exists at lest one $0\leq j\leq m^{\ell}$ such that
$$\varphi(x_j,\cdots,x_{j+\ell-1})\geq \alpha'_{\min}>\alpha_{\min}$$ where $\alpha'_{\min}$ is the second smallest value of $\varphi$ over $S^\ell$, i.e., $\alpha'_{\min}=\min_{a\in S^\ell}\{\varphi(a)\ :\ \varphi(a)>\alpha_{\min} \}$.

We deduce from the above discussions that for any $(x_j)_{j=0}^\infty\in \Sigma_m$ and any $k\geq0$ we have
$$\sum_{j=k}^{k+m^{\ell}}\varphi(x_j,\cdots,x_{j+\ell-1})\geq m^{\ell}\alpha_{\min }+\alpha'_{\min}=(m^{\ell}+1)\alpha_{\min}+\delta,$$ where we denote $\delta=\alpha'_{\min}-\alpha_{\min}$.
This implies that for any $(x_j)_{j=0}^\infty\in \Sigma_m$ and any $n\geq 1$, we have
\begin{equation}\label{equ prop extrema1}
\sum_{j=0}^{n-1}\varphi(x_j,\cdots,x_{j+\ell-1})\geq n\alpha_{\min}+\left\lfloor\frac{n}{m^{\ell}+1}\right\rfloor \delta.
\end{equation}
Now, we will use the above inequality and Proposition \ref{prop 1}  to show the existence of an $\epsilon>0$ such that $$P'_\varphi(s)\geq \alpha_{\min }+\epsilon,\ \ \forall s\in \R.$$
By  Proposition \ref{prop 1}, we have
\begin{equation}\label{equ prop extrema2}
P'_\varphi(s)=(q-1)^2\sum_{k=1}^\infty\frac{1}{q^{k+1}}\sum_{j=0}^{k-1}\E_{\mu_s}
\varphi(x_j,\cdots,x_{j+\ell-1}).
\end{equation}
We can rewrite the term $\sum_{j=0}^{k-1}\E_{\mu_s}
\varphi(x_j,\cdots,x_{j+\ell-1})$ as
$$\E_{\mu_s}\sum_{j=0}^{k-1}
\varphi(x_j,\cdots,x_{j+\ell-1}).$$

By (\ref{equ prop extrema1}), we have for any $(x_j)_{j=0}^{\infty}\in \Sigma_m$
$$\sum_{j=0}^{k-1}\varphi(x_j,\cdots,x_{j+\ell-1})\geq k\alpha_{\min}+\left\lfloor\frac{k}{m^{\ell}+1}\right\rfloor \delta.$$
As $\mu_s$ is a probability measure, we have
$$\E_{\mu_s}\sum_{j=0}^{k-1}
\varphi(x_j,\cdots,x_{j+\ell-1})\geq k
\alpha_{\min}+\left\lfloor\frac{k}{m^{\ell}+1}\right\rfloor \delta.$$
Substituting this in (\ref{equ prop extrema2}), we get

\begin{eqnarray*}
P'_\varphi(s) & = & (q-1)^2\sum_{k=1}^\infty\frac{1}{q^{k+1}}\left(k
\alpha_{\min}+\left\lfloor\frac{k}{m^{\ell}+1}\right\rfloor \delta\right)\\
& = & \alpha_{\min}+\delta(q-1)^2\sum_{k=1}^\infty\frac{1}{q^{k+1}}\left(\left\lfloor\frac{k}{m^{\ell}+1}\right\rfloor\right).
\end{eqnarray*}
As
$$\delta(q-1)^2\sum_{k=1}^\infty\frac{1}{q^{k+1}}\left(\left\lfloor\frac{k}{m^{\ell}+1}\right\rfloor\right)=\delta(q-1)^2\sum_{k\geq m^{\ell}+1}\frac{1}{q^{k+1}}\left(\left\lfloor\frac{k}{m^{\ell}+1}\right\rfloor\right)>0$$
we have proved the existence of an $\epsilon >0$ such that
$$P'_{\varphi}(s)\geq \alpha_{\min}+\epsilon,\ \ \forall s\in \R.$$
\end{proof}

\subsection{Lower bounds of $\dim_H E(\alpha)$}

First, as an easy application of Proposition \ref{prop 1}, we get the following formula for $\dim_H \P_{\mu_s}$.

\begin{proposition}\label{prop formula dim P_s}
For any $s\in \R$, we have
$$\dim_H\P_{\mu_s}=\frac{1}{q^{\ell-1}}[-s\pv'(s)+\pv(s)].$$
\end{proposition}
\begin{proof}
By Proposition \ref{prop basic formula loc dim}, we have
{\setlength\arraycolsep{2pt}
\begin{eqnarray}
-\frac{\log\P_{\mu_s}([x_1^n])}{n} & = &-\frac{s}{n}\sum_{j=1}^{\lfloor\frac{n}{q^{\ell-1}}\rfloor}\varphi(x_j\cdots x_{jq^{\ell-1}})+\frac{n-\lfloor n/q\rfloor}{n}\log \psi_s^q(\emptyset)
\nonumber\\
& & +\frac{B_{\frac{n}{q}}(x)}{\frac{n}{q}} -\frac{B_n(x)}{n}
\end{eqnarray}}

Applying the law of large numbers to the function $\psi_s$, we get
the $\P_{\mu_s}$-a.e. existence of the following limit
$\lim_{n\to\infty}\frac{B_n(x)}{n}$. So
$$\lim_{n\to\infty}\frac{B_{\frac{n}{q}}(x)}{\frac{n}{q}}
-\frac{B_n(x)}{n}=0, \ \ \P_{\mu_s}-{\rm a.e.}$$ On the other hand,
by Proposition \ref{prop 1} and Theorem~\ref{thm esperence general formula}, 
we have
$$\lim_{n\to \infty}\frac{1}{n}\sum_{j=1}^{\lfloor\frac{n}{q^{\ell-1}}\rfloor}\varphi(x_j\cdots x_{jq^{\ell-1}})=\frac{1}{q^{\ell-1}}\pv'(s).$$
So we obtain that for $\P_{\mu_s}$-a.e. $x\in\Sigma_m$
$$\lim_{n\to\infty}-\frac{\log\P_{\mu_s}([x_1^n])}{n}=\frac{1}{q^{\ell-1}}[-s\pv'(s)+\pv(s)],$$ where we have used the fact that
$$P(s)=(q-1)q^{\ell -2}\log\sum_{j\in S^\ell}\psi_s(j)=(q-1)q^{\ell -2}q\log\psi_s(\emptyset). $$


\end{proof}

By Proposition \ref{prop 1}, Proposition \ref{prop formula dim P_s} and Billingsley's lemma (Lemma \ref{Billingsley}) we get the following lower bound for $\dim_HE(P'_\varphi(s))$.
\begin{proposition}For any $s\in\R$, we have
$$\dim_HE(P'_\varphi(s))\geq \frac{1}{q^{\ell-1}\log m}[-\alpha
P'_\varphi(s)+P_\varphi(s)].$$
\end{proposition}
By the above proposition and Proposition \ref{prop upper bound} we obtain the following theorem about the exact Hausdorff dimension of $\dim_H(\alpha)$ for $\alpha\in (\pv'(-\infty),\pv'(+\infty))$.
\begin{theorem}

{\rm (i)} If $\alpha=P'_\varphi(s_\alpha)$ for some $s_\alpha\in\R$, then
$$\dim_H
E(\alpha)= \frac{1}{q^{\ell-1}\log m}[-P'_\varphi(s_\alpha)
s_\alpha+P_\varphi(s_\alpha)]=\frac{\pv^*(\alpha)}{q^{\ell-1}\log m}.$$

{\rm (ii)} For $\alpha\in (P'_\varphi(-\infty),P'_\varphi(0)]$, we have
$$\dim_HE^+(\alpha)=\dim_HE(\alpha).$$

For  $\alpha\in [P'_\varphi(0),P'_\varphi(+\infty))$, we have
$$\dim_HE^-(\alpha)=\dim_HE(\alpha).$$
\end{theorem}


\medskip
\medskip
\medskip
\subsection{Dimension of level sets corresponding to the extreme points in $L_\varphi$}
So far, we have calculated $\dim_HE(\alpha)$ for $\alpha$ in
 $(P'_\varphi(-\infty),P'_\varphi(+\infty))$. Now we turn to the case
 when $\alpha=P'_\varphi(-\infty)$ or $P'_\varphi(+\infty)$. The aim of this subsection is to prove the following result.
\begin{theorem}\label{thm extrema}
If $\alpha=P'_\varphi(-\infty)$ or $P'_\varphi(+\infty)$, then
$E(\alpha)\neq \emptyset$ and
$$\dim_H E(\alpha)=\frac{P_{\varphi}^*(\alpha)}{q^{\ell-1}\log m}.$$
\end{theorem}
We will give the proof of Theorem \ref{thm extrema} for
$\alpha=P'_\varphi(-\infty)$. The proof for
$\alpha=P'_\varphi(-\infty)$ is similar.

\subsubsection{Accumulation points of $\mu_s$ when $s$ tends to $-\infty$ }
We view the vector $\pi_s$ defined by (\ref{def measure 1}) and the matrix $Q_s$ defined by (\ref{def measure 2})
as functions of $s$ taking values in finite dimensional Euclidean
spaces. As all components of $\pi_s$ and $Q_s$ are non-negative
and bounded by 1, the set $\{(\pi_s,Q_s),s\in \R \}$ is pre-compact in a Euclidean space. So there
exists a sequence $(s_n)_{n\in \N}$ of real numbers with
$\lim_{n\to\infty}s_n=-\infty$ such that the limits
$$\lim_{n\to\infty}\pi_{s_n}, \ \ \ \lim_{n\to\infty} Q_{s_n}$$
exist. Using these limits as initial law and
transition probability, we construct a $(\ell-1)$-step Markov measure which we
denote by $\mu_{-\infty}$. It is clear that the Markov measure $\mu_{s_n}$
corresponding to $\pi_{s_n}$ and $Q_{s_n}$ converges to
$\mu_{-\infty}$ with respect to the weak-star topology.

\begin{proposition} We have
$$\P_{\mu_{-\infty}}(E(P'_\varphi(-\infty)))=1.$$
In particular, $E(P'_\varphi(-\infty))\neq \emptyset$.
\end{proposition}
\begin{proof}
First, we introduce a functional on the space of probability measures which is defined by
$$M(\nu)=(q-1)^2\sum_{k=1}^\infty\frac{1}{q^{k+1}}\sum_{j=0}^{k-1}\E_\nu \varphi(x_j,\cdots,x_{j+\ell-1}).$$
The function $\nu\mapsto M(\nu)$ is continuous, just because $\nu\mapsto \E_\nu \varphi(x_j,\cdots,x_{j+\ell-1})$ is continuous for all $j$.

What we have to show is that for $\P_{\mu_{-\infty}}$-a.e. $x\in \Sigma_m$ we have
$$\lim_{n\to\infty}\frac{1}{n}\sum_{k=1}^n\varphi(x_k,\cdots,x_{kq^{\ell-1}})=\pv'(-\infty).$$
By Theorem \ref{thm esperence general formula},  for $\P_{\mu_{-\infty}}$-a.e. $x\in \Sigma_m$ the limit in the left hand side of the above equation equals to $M(\mu_{-\infty})$. As $M$ is continuous and
 $\mu_{s_n}$ converges to $\mu_{-\infty}$ when $n\to \infty$, we deduce that$$\lim_{n\to \infty}M(\mu_{s_n})=M(\mu_{-\infty}).$$
By Proposition \ref{prop 1}, we know that $$M(\mu_{s_n})=\pv'(s_n).$$
So $$M(\mu_{-\infty})=\lim_{n\to \infty}\pv'(s_n).$$
By Theorem \ref{thm convexity}, the map $s\to \pv'(s)$ is increasing, thus we deduce that the above limit exists and
$$M(\mu_{-\infty})=\pv'(-\infty).$$
This implies the desired result.

\end{proof}

We have the following formula for $\dim_H\P_{\mu_{-\infty}}$.

\begin{proposition} We have
$$\dim_H\P_{\mu_{-\infty}}=\lim_{s\to-\infty}\frac{[-s\pv'(s)+\pv(s)]}{q^{\ell-1}\log m}=\frac{P_{\varphi}^*(\pv'(-\infty))}{q^{\ell-1}\log m}.$$

\end{proposition}
\begin{proof}
By Proposition \ref{prop loc dim}, we know for any probability measure $\nu$ we have
$$\dim_H\P_\nu=\frac{(q-1)^2}{\log m}\sum_{k=1}^{\infty}\frac{H_k(\nu)}{q^{k+1}}.$$
As the series in the right hand side converges uniformly on $\nu$, the map $\nu\to \dim_H\P_{\nu}$ is continuous.
Since $\mu_{s_n}$ converges to $\mu_{-\infty}$ when $n\to \infty$, we deduce that
$$\lim_{n\to\infty}\dim_H\P_{\mu_{s_n}}=\dim_H\P_{\mu_{-\infty}}.$$
By Proposition \ref{prop formula dim P_s}, we have
$$\dim_H\P_{\mu_s}=\frac{[-s\pv'(s)+\pv(s)]}{q^{\ell-1}\log m}.$$
The derivative of the map $s\to\dim_H\P_{\mu_s}$ is
$$\frac{d}{ds}\dim_H\P_{\mu_s}=\frac{-s\pv''(s)}{q^{\ell-1}\log m}.$$
As $\pv(s)$ is convex on $\R$, $\pv''(s)$ is non-negative, so for $s\leq 0$ the map $$s\to\dim_H\P_{\mu_s}$$ is increasing. Thus
$$\dim_H\P_{\mu_{-\infty}}=\lim_{n\to\infty}\dim_H\P_{\mu_{s_n}}=\lim_{s\to-\infty}\frac{[-s\pv'(s)+\pv(s)]}{q^{\ell-1}\log m}.$$
\end{proof}

\begin{proposition}
$$\dim_H E(\pv'(-\infty))=\frac{P_{\varphi}^*(\pv'(-\infty))}{q^{\ell-1}\log m}.$$
\end{proposition}
\begin{proof}
By the last two propositions and Billingsley's lemma, we get
$$\dim_HE(\pv'(-\infty))\geq \frac{\pv^*(\pv'(-\infty))}{q^{\ell-1}\log m}.$$
We now show the reverse inequality. By the definition of $E^+(\alpha)$, we have
$$E(\pv'(-\infty)))\subset \bigcap_{\alpha\in (\pv'(-\infty),\pv'(0)]}E^+(\alpha)= \bigcap_{s\leq 0}E^+(\pv'(s)).$$
So $$\dim_HE(\pv'(-\infty)))\leq \dim_HE^+(\pv'(s))=\dim_HE(\pv'(s))=\dim_H\P_{\mu_s},\ \ \forall s\leq0.$$
Now as $s\to \dim_H\P_{\mu_s}$ is increasing we deduce that
$$\dim_HE(\pv'(-\infty)))\leq\lim_{s\to-\infty}\dim_H\P_{\mu_s}=\frac{P_{\varphi}^*(\pv'(-\infty))}{q^{\ell-1}\log m}.$$
\end{proof}


\section{The invariant part of $E(\alpha)$}
\bigskip

From classical dynamical system point of view, the set $E(\alpha)$ is not invariant and its dimension can not be
described by invariant measures supported on  it, as we shall see. Let us first examine the largest dimension of ergodic measures supported on the set $E(\alpha)$.

Here we can consider a more general setting. Let $f_1, f_2, \cdots, f_\ell$ be real functions defined on $\Sigma_m$. Let
\begin{equation}\label{MEA}
      M_{f_1, \cdots, f_\ell}(x)=
\lim_{n\rightarrow
\infty}\frac{1}{n}\sum_{k=1}^{n}f_{1}(T^{k}x)f_{2}(T^{2k}x)\cdots
f_{\ell}(T^{\ell k}x)
      \end{equation}
      if the limit exists. In this section, for a real number $\alpha$, we
      define $$E(\alpha)
      =\{x \in \Sigma_m: M_{f_1, \cdots, f_\ell}(x) = \alpha\}.$$
In order to describe the invariant part of $E(\alpha)$, we introducing the so-called invariant spectrum:
$$
    F_{\rm inv}(\alpha) = \sup \left\{\dim \mu: \mu \ \mbox{\rm ergodic}, \mu(E(\alpha))=1\ \right\}.
$$

In general,
 $F_{\rm inv}(\alpha)$ is  smaller than $\dim E(\alpha)$. It is even possible that no ergodic
measure is supported on $E(\alpha)$.

\medskip

\begin{theorem}\label{mixing} Let $\ell=2$. Let $f_1$ and $f_2$ be two H\"{o}lder continuous functions on $\Sigma_m$.
If $E(\alpha)$ supports an ergodic measure, then
$$
F_{\rm inv}(\alpha)
= \sup \left\{\dim \mu: \mu \ \mbox{\rm ergodic}, \int f_1 d\mu \int f_2 d\mu = \alpha \ \right\}.
$$
    \end{theorem}

 \begin{proof}   Let $\mu$ be an ergodic measure such that $\mu(E(\alpha))=1$. Then
\begin{eqnarray*}
   \alpha
   & = & \lim_{n\to \infty}\frac{1}{n}\sum_{k=1}^n \mathbb{E}_\mu [f_1(T^k x)f_2(T^{2k} x)]\\
   & = & \lim_{n\to \infty}\frac{1}{n} \sum_{k=1}^n \mathbb{E}_\mu [f_1( x)f_2(T^{k} x)]\\
  & = & \mathbb{E}_\mu [f_1(x) M_{f_2}(x)]
\end{eqnarray*}
where the first and third equalities are due to Lebesgue convergence theorem and the second one is due to
the invariance of $\mu$. Since $\mu$ is ergodic, $M_{f_2}(x)=\mathbb{E}_\mu f_2$ for $\mu$-a.e. $x$.
So, $\alpha = \mathbb{E}_\mu f_1 \mathbb{E}_\mu f_2$. It follows that
$$
    F_{\rm inv}(\alpha) \le \sup \left\{\dim \mu: \mu \ \mbox{\rm ergodic}, \mathbb{E}_\mu f_1 \mathbb{E}_\mu f_2 = \alpha \ \right\}.
$$
To obtain the inverse inequality, it suffices to observe  from standard higher--dimensional multifractal analysis for H\"older continuous functions that the above supremum is attained by a
Gibbs measure $\nu$ which is mixing and that the mixing property implies
$M_{f_1,f_2}(x) = \mathbb{E}_\nu f_1 \mathbb{E}_\nu f_2$ $\nu$-a.e..
\end{proof}

\begin{remark}\label{remark mixing}
In the above theorem, the assumption that $\mu$ is ergodic can be relaxed to $\mu$ is invariant.  In fact, if $\nu$ is an invariant measure
such that $\nu(E(\alpha)) = 1$. Then, by the ergodic decomposition theorem
and the corresponding decomposition of entropy (a theorem due to
Jacobs), there is an ergodic measure $\mu$ such that $\mu(E(\alpha)) = 1$ and
$h_\nu\le h_\mu$. When $\ell\ge 3$, the result in above theorem remains true if we replace ``ergodic'' by ``multiple mixing'', i.e.
$$F_{\rm mix}(\alpha)=\sup \left\{\dim \mu: \mu \ \mbox{\rm multiple mixing},\ \mathbb{E}_\mu f_1 \cdots \mathbb{E}_\mu f_\ell= \alpha \ \right\},$$
where
$$
    F_{\rm mix}(\alpha) = \sup \left\{\dim \mu: \mu \ \mbox{\rm multiple\ mixing}, \mu(E(\alpha))=1\ \right\}.
$$

\end{remark}

Here is a remarkable corollary of the above theorem. Assume that $f_1=f_2=f$.
If $\mu(E(\alpha)=1$ for some ergodic measure $\mu$, then we must have
$$
  \alpha = \left(\int f d \mu\right)^2 \ge 0.
$$
There are examples of $f$ taking  negative value such that for some $\alpha <0$ we have $\dim E(\alpha)>0$. However, the theorem together with the remark shows that there is no invariant measure with positive dimension supported by $E(\alpha)$. See Example 2 below.

In the proof of the theorem, the fact that $M_{f_1}$ is almost constant plays an important role.
It is not the case for $M_{f_1, f_2}$. So we can not generalize the theorem to $\ell =3$.

For $f_1, f_2 \in L^2(\mu)$ where $\mu$ is an ergodic measure, Bourgain proved that
$M_{f_1, f_2}(x)$ exists for $\mu$-almost all $x$. The limit is in general not constant, but can be
written by the Kronecker factor $(Z, m, S)$, which is considered as a rotation on a compact abelian group $Z$.
Let $\pi$ be the factor map. Let
$$
   \tilde{f}_i = \mathbb{E}(f_i|Z).
$$
Then $\mu$-almost surely
$$
    M_{f_1, f_2}(x) = \int_Z \tilde{f}_1 (\pi(x) +z ) \tilde{f}_1 (\pi(x) + 2z ) dm(z).
$$
Then it is easy to deduce that $M_{f_1, f_2}(x)$ is $\mu$-almost surely constant if and only if
$$
   \forall \gamma \in \widehat{Z} \ \mbox{\rm with}\ \gamma \not=1, \widehat{\tilde{f}_1}(\gamma)
   \widehat{\tilde{f}_2}(\gamma^2) =0.
$$
This condition is extremely strong if $\mu$ is not weakly mixing. In other words, when $\ell
= 3$, it would be exceptional that
$E(\alpha)$ carries an ergodic measure which is not weakly mixing. When $\mu$ is mixing,
we have $M_{f_1, f_2}(x) = \int f_1 d\mu \int f_2 d\mu$ for $\mu$-almost all $x$.

For three or more functions, the existence of the almost everywhere limit  $M_{f_1, f_2, \cdots, f_\ell}$
is not yet proved. But the $L^2$-convergence is proved by Host and Kra \cite{HK}.
The limit can be written as
a similar integral, but the integral is taken over a nilmanifold of order $2$ \cite{BHK}.

Let us also remark that the supremum in the theorem is also equal to the dimension of the $\alpha$-level set of
$$
    \lim_{n\to \infty} \frac{1}{n^2}\sum_{1\le j, k\le n} f_1(T^j x) f_2(T^k x).
$$
See \cite{FLP}. Also see \cite{FSW_V}, where general $V$-statistics are  studied.

\section{Examples}

\bigskip

The motivation of the subject initiated in \cite{FLM} is the
following example. The Riesz product method used in \cite{FLM}
doesn't work for this case. However Theorem \ref{thm principal} does.

\begin{example}\label{11}
Let $q=2$, $m=2$, $\ell=2$ and $\varphi$ the potential given by $\varphi(x,y)=x_1y_1$ with $x=(x_i)_{i=1}^\infty ,y=(y_i)_{i=1}^\infty\in \Sigma_2$. So $$\left[\varphi(i,j)\right]_{(i,j)\in \{0,1\}^2}=
\left[ \begin{array}{cc}
0 & 0 \\
0 & 1
\end{array} \right].
$$
\end{example}

The system of equations (\ref{transer_equation}) in this case
becomes \begin{eqnarray*} \psi_s(0)^2 & = & \psi_s(0)+\ \ \psi_s(1),\\
\psi_s(1)^2& = & \psi_s(0)+e^s\psi_s(1).
\end{eqnarray*}
Fix $s\in\R$. By solving an fourth order algebraic equation, we get
the unique positive solution of the above system:
\begin{eqnarray*}
\psi_s(0) &= &\frac{1}{6}a(s)+\frac{2/3-2e^s}{a(s)}+\frac{2}{3},\\
\psi_s(1) &= &\psi_s(0)^2-\psi_s(0),
\end{eqnarray*}
where
$$ a(s)=\left(100-36
e^s+12\sqrt{69-54e^s-3e^{2s}-12e^{3s}}\right)^\frac{1}{3}.
$$
Recall that the pressure function is equal to
$$P_\varphi(s)=\log(\psi_s(0)+\psi_s(1)).$$

The minimal and maximal values of $\varphi$ are $0$ and $1$, which
are respectively attained  by the   sequences
$(x_j)_{j=0}^\infty=(0)^\infty$ and $(y_j)_{j=0}^\infty=(1)^\infty$
in the sense of
$$\varphi(x_j,x_{j+1})=0,\ \ \varphi(y_j,y_{j+1})=1,\ \ \forall j\geq 0.$$
Then by Theorem \ref{critere cercle}, we have
$$P'_\varphi(-\infty)=0,\ \ P'_\varphi(+\infty)=1.$$

Therefore, according to Theorem \ref{thm principal}, for any
$\alpha\in [0,1]$ we have
$$\dim_HE(\alpha)=\frac{-\alpha s_\alpha+P_\varphi(s_\alpha)}{2\log 2},$$
where $s_\alpha$ is the unique real such that $P'_\varphi(s_\alpha)=\alpha.$

We now consider the invariant spectrum of $E(\alpha)$.  As
$\varphi(x,y)=f(x)f(y)$ with $f(x)=x_1$,
 by Theorem
~\ref{mixing}, we have
$$F_{\rm inv}(\alpha)=\sup\left\{\frac{h_\mu}{\log 2}\ :\ \mu\in \mathcal{M}_{\rm inv}(\Sigma_2),\
\int x_1 d\mu=\sqrt{\alpha} \right\}.$$ It is well known (see
\cite{FFW}) that the right hand side, which is attained by a
Bernoulli measure, is equal to
$$H(\sqrt{\alpha})=-\sqrt{\alpha}\log_2 \sqrt{\alpha}-(1-\sqrt{\alpha})\log_2 (1-\sqrt{\alpha}).$$
So $$F_{\rm inv }(\alpha)\ =\ H(\sqrt{\alpha}).$$

See Figure \ref{figure 1}  for the graphs of the spectra
$\alpha\mapsto \dim_HE(\alpha)$ and $\alpha\mapsto L_{\rm
inv}(\alpha)$. We remark that, except at the extremal points
($\alpha=1/4$ or $1$), we have a strict inequality $F_{\rm
inv}(\alpha)<\dim_HE(\alpha)$. This shows that the invariant part of
$E(\alpha)$ is much smaller than $E(\alpha)$ itself. This is
different of the classical ergodic theory ($\ell=1$) where in
general we have $F_{\rm inv}(\alpha)=\dim_HE(\alpha)$ for all
$\alpha$ and actually $E(\alpha)$ is invariant.
\begin{figure}
\centering
\includegraphics[width=6.8cm]{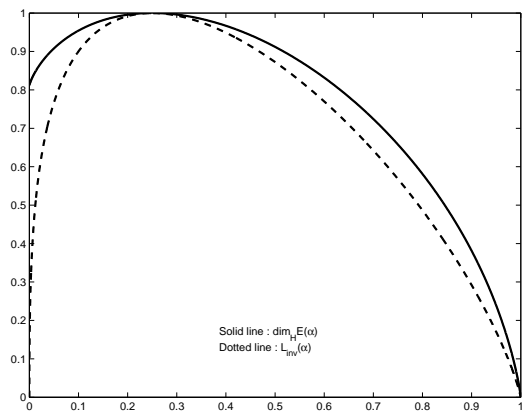}
\caption{The graphs of the spectrum $\alpha\mapsto \dim_HE(\alpha)$ and $\alpha\mapsto F_{\rm inv}(\alpha)$ (Example 1).}
\label{figure 1}
\end{figure}

The following example is a special case of a situation studied in \cite{FLM}. So, the
result is not new. Applying Theorem~\ref{thm principal} only
provides a second way to get it. But when we compare its invariant
spectrum with its multifractal spectrum we will discover a new
phenomenon--there is "no" invariant part in $E(\alpha)$ for some
$\alpha$.

\begin{example}
Let $q=2$, $m=2$, $\ell =2$ and $\varphi$ be the potential given by
$\varphi(x,y)=(2x_1-1)(2x_2-1)$. So
$$\left[\varphi(i,j)\right]_{(i,j)\in \{0,1\}^2}=
 \left[ \begin{array}{rr}
  1 & -1 \\
-1 &  1
\end{array} \right].
$$
\end{example}

The system of equations (\ref{transer_equation}) in this case
reduces to
\begin{eqnarray*}
\psi_s(0)^2&=&e^{s}\psi_s(0)+e^{-s}\psi_s(1),\\
\psi_s(1)^2&=&e^{-s}\psi_s(0)+e^{s}\psi_s(1).
\end{eqnarray*}
Because of the symmetry of $\varphi$, it is easy to find the  unique
positive solution of the  system:
$$\psi_s(0)\ =\ \psi_s(1)\ =\ e^{s}+e^{-s}.$$
Thus we get the pressure function
$$P_\varphi(s)=\log(\psi_s(0)+\psi_s(1))=\log 2+\log (e^{s}+e^{-s}).$$
It is evident that
$$P'_\varphi(s)\ =\  \frac{e^{s}-e^{-s}}{e^{s}+e^{-s}}.$$
and
$$P'_\varphi(-\infty)=-1,\ \quad P'_\varphi(+\infty)=1.$$
So, by Theorem \ref{thm principal}, we have $L_{\varphi}=[-1,1]$,
and for any $\alpha\in [-1,1]$ we have
$$\dim_HE(\alpha)=\frac{-\alpha s_\alpha+P_\varphi(s_\alpha)}{2\log 2},$$
where $s_\alpha$ is such that
$$ \frac{e^{s_\alpha}-e^{-s_\alpha}}{e^{s_\alpha}+e^{-s_\alpha}}=\alpha.$$

We now consider the invariant spectrum of $E(\alpha)$.  We have
$\varphi(x,y)=f(x)f(y)$ with $f(x)=2x_1-1$, then by Theorem
~\ref{mixing}, we have
$$F_{\rm inv}(\alpha)=\sup\left\{\frac{h_\mu}{\log 2}\ :\ \mu\in \mathcal{M}_{\rm inv}(\Sigma_2),\
\left(\int (2x_1-1) d\mu\right)^2=\alpha \right\}.$$ We see that we
must assume $\alpha\ge 0$. As $\int (2x_1-1) d\mu=2 \int x_1
d\mu-1$, the condition $\left(\int (2x_1-1) d\mu\right)^2=\alpha$
means $\int x_1 d\mu= \frac{1}{2}(1 \pm \sqrt{\alpha})$. The above
supremum is attained by a Bernoulli measure determined by the
probability vector $((1+\sqrt{\alpha})/2, (1- \sqrt{\alpha})/2)$.
In other word,
 $$F_{\rm inv }(\alpha)\ =\ H\left(\frac{1+\sqrt{\alpha}}{2}\right)$$
 where $H(x) = -x \log_2 x - (1-x) \log_2 (1-x)$.

See Figure \ref{figure 2} for the graphs of the spectra
$\alpha\mapsto \dim_HE(\alpha)$ and $\alpha\mapsto F_{\rm
inv}(\alpha)$. We see that, except at the extremal point $\alpha=0$,
we have $F_{\rm inv}(\alpha)<\dim_HE(\alpha)$. Moreover, for $-1
\leq \alpha<0$, we have $F_{\rm inv }(\alpha)=0$.  That is to say,
there is no invariant 
measure with positive dimension sitting on
$E(\alpha)$ for $-1\leq \alpha<0$. 
But $\dim_HE(\alpha)>0$.


\begin{figure}
\centering
\includegraphics[width=6.8cm]{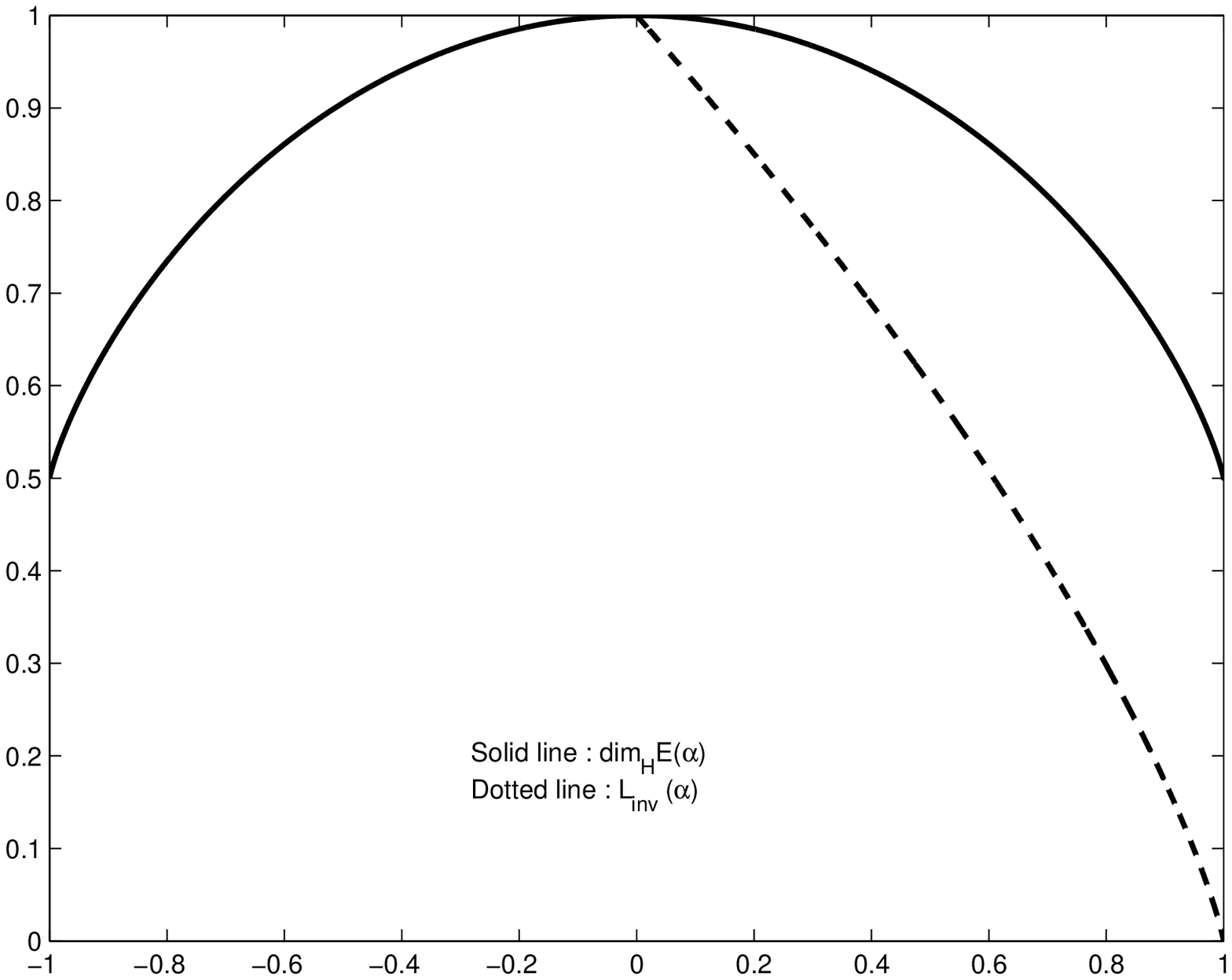}
\caption{The graph of the spectrum $\alpha\mapsto \dim_HE(\alpha)$ (Example 2).}
\label{figure 2}
\end{figure}

\medskip

The following example presents a case where the $L_\varphi$ is
strictly contained in the interval $[\alpha_{\min}, \alpha_{\max}]$.

\begin{example}
Let $q=2$, $m=2$, $\ell=2$ and $\varphi$ be  the potential given by
$\varphi(x,y)=y_1-x_1$. In other words,
$$\left[\varphi(i,j)\right]_{(i,j)\in \{0,1\}^2}= \left[
\begin{array}{cc}
0 & -1 \\
1 & 0
\end{array} \right].
$$
\end{example}

The system of equations (\ref{transer_equation}) in this case
reduces to \begin{eqnarray*} \psi_s(0)^2
&=&\ \ \ \ \psi_s(0)+e^s\psi_s(1),\\
 \psi_s(1)^2&=&e^{-s}\psi_s(0)+\ \ \psi_s(1). \end{eqnarray*}
It is easy to find the unique positive solution of the system:
$$\psi_s(0)=1+e^{\frac{s}{2}},\ \quad \psi_s(1)=1+e^{-\frac{s}{2}}.$$
The pressure function is then given by
$$P_\varphi(s)=\log(\psi_s(0)+\psi_s(1))=\log(2+e^{\frac{s}{2}}+e^{-\frac{s}{2}}).$$
So
$$P'_\varphi(s)= \frac{1}{2} \frac{e^{s/2}-e^{-s/2}}{2+e^{s/2}+e^{-s/2}},$$
and
$$P'_\varphi(-\infty)=-\frac{1}{2}, \ \quad P'_\varphi(+\infty)=\frac{1}{2}.$$
Remark that in this case we have
$$\alpha_{\min}<P'_\varphi(-\infty)<
P'_\varphi(+\infty)<\alpha_{\max}.
$$

By Theorem \ref{thm principal}, we have $L_{\varphi}=[-1/2,1/2]$,
and for any $\alpha\in [-1/2,1/2]$ we have
$$\dim_HE(\alpha)=\frac{-\alpha s_\alpha+P_\varphi(s_\alpha)}{2\log 2},$$
where $s_\alpha$ is the solution of
$$
\frac{e^{s_\alpha/2}-e^{-s_\alpha/2}}{2+e^{s_\alpha/2}+e^{-s_\alpha/2}}=2
\alpha.$$

We now consider the invariant spectrum of $E(\alpha)$.  We have
$\varphi(x,y)=f(y)-f(x)$ with $f(x)=x_1$. By Lebesgue convergence theorem, for any $\alpha\in \R$ such that there exists an invariant measure $\mu$ with $\mu(E(\alpha))=1$ we have
$$\alpha=\lim_{n\to \infty}\frac{1}{n}\sum_{k=1}^n\E_\mu(x_{2k}-x_k)=\lim_{n\to \infty}\frac{1}{n}\sum_{k=1}^n\left(\E_\mu(x_{2k})-\E_\mu(x_{k})\right)=0.$$ (The last equality is due to the invariance of $\mu$). This means that the only $\alpha$ such that there is an invariant measure with positive dimension sitting on $E(\alpha)$ is $\alpha =0$. The invariant spectrum then degenerates to one point. We have $F_{\rm inv}(0)=1$.

See Figure \ref{figure 3} for the graph of the spectrum
$\alpha\mapsto \dim_HE(\alpha)$.
\begin{figure}
\centering
\includegraphics[width=6.8cm]{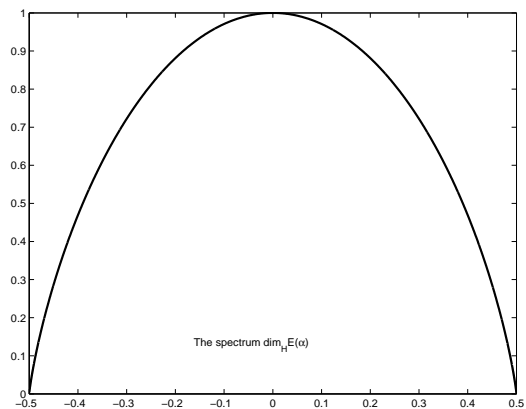}
\caption{The graph of the spectrum $\alpha\mapsto \dim_HE(\alpha)$
(Example 3).} \label{figure 3}
\end{figure}

\medskip

We can easily solve the system (\ref{transer_equation}) for a class
of symmetric functions described in the following example. The
example 2 is a special case.

\begin{example}
Let $\ell=2$, $q\geq 2$ and $m\geq2$. Let
$\varphi=\left[\varphi(i,j)\right]_{(i,j)\in \{0,\cdots,m-1\}^2}$ be
a  potential considered as a matrix. Suppose that each row of the
matrix  is a permutation of the first row.
\end{example}

Recall the system of equations (\ref{transer_equation}):
$$\psi_s(i)^q\ =\ \sum_{j=0}^{m-1}e^{s\varphi(i,j)}\psi_s(j),\ \ i\in \{0,\cdots,m-1\}.$$
It is straightforward to verify that the constant vector
$(a,\cdots,a)$, with
$$a=\left(\sum_{j=0}^{m-1}e^{s\varphi(1,j)}\right)^{\frac{1}{q-1}},$$
is the unique positive solution of the above system (see Theorem
\ref{existence-unicity trans-equ}). The pressure function is then
given by
$$P_\varphi(s) 
=\log \sum_{j=0}^{m-1}e^{s\varphi(1,j)}+(q-1)\log m.$$
We have
$$P'_\varphi(s)=\frac{\sum_{j=0}^{m-1}e^{s\varphi(1,j)}\varphi(1,j)}{\sum_{j=0}^{m-1}e^{s\varphi(1,j)}}.$$
Then
$$\lim_{s\to-\infty}P'_\varphi(s)=\lim_{s\to-\infty}\frac{\sum_{j=0}^{m-1}e^{s(\varphi(1,j)-\alpha_{\min})}\varphi(1,j)}{\sum_{j=0}^{m-1}e^{s(\varphi(1,j)-\alpha_{\min})}}=\alpha_{\min}=\min_j \varphi(1,j).$$
Similarly, we have
$$\lim_{s\to+\infty}P'_\varphi(s)=\alpha_{\max}=\max_j \varphi(1,j).$$
By  the  hypothesis of symmetry on $\varphi$, it is easy to see that
there exist sequences $(x_j)_{j=0}^\infty$ and $(y_j)_{j=0}^\infty
\in \Sigma_m$ such that
$$\varphi(x_j,x_{j+1})=\alpha_{\min},\ \quad
\varphi(y_j,y_{j+1})=\alpha_{\max},\ \forall j\geq 0.$$ Therefore,
by Theorem \ref{thm principal},
$L_{\varphi}=[\alpha_{\min},\alpha_{\max}]$, and for any $\alpha\in
[\alpha_{\min},\alpha_{\max}]$ we have
$$\dim_HE(\alpha)=\frac{-\alpha s_\alpha+P_\varphi(s_\alpha)}{2\log m},$$
where $s_\alpha$ is the solution of
$$\frac{\sum_{j=0}^{m-1}e^{s_\alpha\varphi(1,j)}\varphi(1,j)}{\sum_{j=0}^{m-1}e^{s_\alpha\varphi(1,j)}}=\alpha.$$

The invariant spectrum: For $\alpha \in [\alpha_{\min},
\alpha_{\max}]$, the invariant spectrum is attained by a Markov
measure. That is to say
$$
   F_{\rm inv}(\alpha) = \sup \left\{ -\sum_{0\le i, j\le m-1} \pi_i p_{i, j} \log_m
   p_{i, j}:
   \sum_{0\le i, j\le m-1} \varphi(i, j) \pi_i p_{i, j} = \alpha \right\}
$$
where $P=(p_{i, j})$ is a stochastic matrix and $\pi=(\pi_0,\cdots,
\pi_{m-1})$ is an invariant probability vector of $P$, i.e. $\pi
P=\pi$.

\bigskip

In the next example we show that in general the invariant spectrum can be strictly larger than the mixing spectrum for some level set $E(\alpha)$.

\begin{example}
Let $m\geq 2$. Consider two functions $f$ and $h$ on $\Sigma_m$ defined by
\[
f(i)=\begin{cases} 1 & 0\le i <m-1\\ 2 & i=m-1 \end{cases} \qquad h(i)=\begin{cases} -2 & 0\le i <m-1\\ 1 & i=m-1 \end{cases}.
\]
Consider the level set
\[
E(0)=\left\{x\in\Sigma_m\, :\, \lim_{n\to\infty}\frac1n\sum_{k=1}^nf(x_k)f(x_{2k})h(x_{3k})=0\right\}.
\]
(That means $\phi(x,y,z)=f(x)f(y)h(z)$).
We claim that $F_{\rm mix}(0)<F_{\rm inv}(0)$ for $m\geq 49$.
\end{example}

Let $\delta_j$ denotes the Dirac measure at $j\in \{0,1,\cdots,m-1\}$. Let
 $$\nu=\frac{1}{m-1}\sum_{j=0}^{m-2}\delta_j.$$ We note that $\nu$ restricted on $\Sigma_{m-1}$ gives rise to the measure of maximal dimension on $\Sigma_{m-1}$. We consider a probability measure on $\Sigma_m$ defined by $$\mu=\frac{1}{2}\mu_1+\frac{1}{2}\mu_2,$$ where
 $$\mu_1([x_1x_2\cdots x_n])=\prod_{k=0}^{\lfloor\frac{n-1}2\rfloor}\delta_{m-1}(x_{2k+1})\cdot\prod_{k=1}^{\lfloor\frac{n}2\rfloor}\nu(x_{2k}),$$
 and
 $$\mu_2([x_1x_2\cdots x_n])=\prod_{k=0}^{\lfloor\frac{n-1}2\rfloor}\nu(x_{2k+1})\cdot\prod_{k=1}^{\lfloor\frac{n}2\rfloor}\delta_{m-1}(x_{2k}).$$

Note that $T^{-1}\circ \mu_1=\mu_2$ and $T^{-1}\circ \mu_2=\mu_1$. So $\mu$ is shift invariant.
The measure $\mu$ sits on the  set $A=A_1\bigcup A_2$ where
\[
A_1=\left\{x\in\Sigma_m\, :\, x_{2k+1}=m-1,\, x_{2k}\ne m-1,\, k\in\N\right\},
\]
\[
A_2=\left\{x\in\Sigma_m\, :\, x_{2k}=m-1,\, x_{2k+1}\ne m-1,\, k\in\N\right\}.
\]
Actually $\mu_1(A_1)=1$ and $\mu_2(A_2)=1$ and the sets $A_1$ and $A_2$ are disjoint.

We claim that $\mu$ is ergodic but not mixing. To see that $\mu$ is not mixing, we only need to observe that $T^{-1}A_1=A_2$ and $T^{-1}A_2=A_1$. From this and that $A_1$ and $A_2$ are disjoint we deduce that
$$\mu\left(T^{-2k}A_1\cap A_2\right)=0,\ \ \forall k\in \N. $$
This implies that $\mu$ is not mixing. The ergodicity of $\mu$ with respect to $T$ is due to the fact that $\mu_1$ and $\mu_2$ are ergodic with respect to $T^2=T\circ T$ and that they are  supported
by disjoint sets.

For every  $x\in A_1$ we have
$$
\lim_{n\to\infty}\frac1n\sum_{k=1}^nf(x_k)f(x_{2k})h(x_{3k})=\lim_{n\to\infty}\frac1n\left(\sum_{k-{\rm even}}+\sum_{k-{\rm odd}}\right)=\frac12\left(-2+2\right)=0
$$
and for every $x\in A_2$ we have
$$
\lim_{n\to\infty}\frac1n\sum_{k=1}^nf(x_k)f(x_{2k})h(x_{3k})=\lim_{n\to\infty}\frac1n\left(\sum_{k-{\rm even}}+\sum_{k-{\rm odd}}\right)=\frac12\left(4-4\right)=0.
$$
Hence, $\mu(E(\alpha)=1$. We note that
\[
\int_{\Sigma_m}f\, d\mu\cdot\int_{\Sigma_m}f\, d\mu\cdot\int_{\Sigma_m}h\, d\mu=\left(\frac{3}{2}\right)^2\cdot\left(-\frac{1}{2}\right)=-\frac98<0.
\]

Let us compute the dimension of $\mu$ by computing the local entropy at typical points. If $x\in A$ then
\[
\mu([x_1\cdots x_{2n}])=(m-1)^{-n}.
\]
Since $\mu(A)=1$ this implies that $\dim_H\mu=\frac12\log(m-1)$. So that $F_{\rm inv}(0)\geq \frac{1}{2}\log (m-1)$.
On the other hand, by Theorem \ref{mixing} and Remark \ref{remark mixing}, we have
\[
F_{\rm mix}(0)=\sup\left\{h_\mu \, :\, \mu-\text{multiple\ mixing, }\, \int_{\Sigma_m}h\, d\mu=0\right\}
\]
since $f$ is strictly positive. From standard multifractal analysis we know that the supremum is attained by a Bernoulli measure and
\begin{align*}
F_{\rm mix}(0)&=\max_{p_i\ge 0}\left\{-\sum_{i=0}^{m-1}p_i\log p_i\, :\, p_0+\cdots +p_{m-2}=\frac13, p_{m-1}=\frac23\right\}\\&
=\frac13\log(m-1)+\frac13\log3+\frac23\log\frac32.
\end{align*}
If $m>48$ we conclude $F_{\rm inv}(0)>F_{\rm mix}(0)$.

\medskip

\section{Remarks and Problems}

{\em Multiplicatively invariant sets.} The first basic example ({\bf Example 1} above) which motivated our study leads to the set
$$
X_2 = \{(x_k)_{k\ge 1} \in \Sigma_2: \forall k\ge 1,  x_k x_{2k} = 0 \}
$$
which was introduced in \cite{FLM}. It is known to F\"{u}rstenberg \cite{Furstenberg0} that any shift-invariant closed set
has its  Hausdorff dimension equal to its Minkowski (box-counting) dimension. Unfortunately the closed set $X_2$ is not shift-invariant.
Its Minkowski dimension was computed by Fan, Liao and Ma \cite{FLM} and its Hausdorff dimension was computed by Kenyon, Peres and Solomyak \cite{KPS}.
The results show that the Hausdorff dimension is smaller than the Minkowski dimension.
Recall that
$$
  \dim_M X_2 = 0.82429..., \quad \dim_H X_2 = 0.81137...
$$
As observed by Kenyon, Peres and Solomyak,
the set $X_2$ is invariant under the action of the semigroup $\mathbb{N}$ in the sense that $T_r X_2 \subset X_2$ for all $r\in \mathbb{N}$
where $T_r$ is defined by
$$
     x=(x_k)_{k\ge 1} \mapsto T_r x = (x_{rk})_{k\ge 1}.
$$
As observed by Fan, Liao and Ma, we have the decomposition
$$
     \mathbb{N} = \bigsqcup_{i: \rm odd} i\Lambda
$$
where $\Lambda =\{1, 2, 2^2, 2^3, \cdots\}$ is the (multiplicative) sub-semigroup generated by $2$. This is one of the key point
in the present study. A similar decomposition holds for semigroups generated by a finite number of prime numbers. Using this decomposition,
Peres, Schmeling, Solomyak and Seuret \cite{PSSS} computed the Hausdorff dimension and the Minkowski dimension of sets like
$$
X_{2,3} = \{(x_k)_{k\ge 1} \in \Sigma_2: \forall k\ge 1,  x_k x_{2k} x_{3k}= 0 \}.
$$
This is an important step.

{\em A generalization.}
Combining the ideas in \cite{PSSS} and those in the present paper, we can study the following limit
$$
     \lim_{n \to \infty} \frac{1}{n} \sum_{k=1}^n \varphi(x_k, x_{2k}, x_{3k}).
$$
See \cite{Wu}. Notice that the computation in this case are more involved. Also notice that, by chance, the Riesz product method
used in \cite{FLM} is well adapted to the study of the special limit
$$
     \lim_{n \to \infty} \frac{1}{n} \sum_{k=1}^n (2x_k-1) (2 x_{2k} -1)\cdots (2 x_{\ell k}-1)
$$
where $\ell \ge 2$ is any integer.

{\em Vector valued potential.} We indicate here how to extend our results to vector valued potentials. First, let $\varphi, \gamma$ be $2$ functions   defined on $S^\ell$ taking real values. Instead of considering the transfer operator $\mathcal{L}_s$ as defined in (\ref{transer-operator}), we consider the following one.
$$
\mathcal{L}_s \psi (a)
= \sum_{j \in S} e^{s \varphi(a, j)+\gamma(a, j)}
\psi (Ta, j),\ a\in S^{\ell-1},\ s\in \R.
$$
Still by Theorem \ref{existence-unicity trans-equ}, there exists a unique solution to the equation
$$(\mathcal{L}_s \psi)^{\frac{1}{q}}=\psi.$$
Then, we can similarly define the pressure function as indicated in (\ref{transer_equation 2}) and (\ref{pressure function}). We denote this pressure function by $P_{\varphi,\gamma}(s)$.
The arguments with which we proved the analyticity and convexity of $s\mapsto\pv(s)$ can  be also used to prove the same results for $s\mapsto P_{\varphi,\gamma}(s)$.

Let $\underline{\varphi}=(\varphi_1,\cdots,\varphi_d)$ be a function defined on $S^\ell$ taking values in $\R^d$. For $\underline{s}=(s_1,\cdots,s_d)\in \R^d$, we consider the following transfer operator.
$$
\mathcal{L}_{\underline{s}} \psi (a)
= \sum_{j \in S} e^{\langle \underline{s},\underline{\varphi}\rangle}
\psi (Ta, j),\ a\in S^{\ell-1},
$$
where  $\langle \cdot,\cdot\rangle$ denotes the scalar product in $\R^d$. We denote the associated pressure function by $P(\underline{\varphi})(\underline{s})$.
Then, by the above discussion, for any vectors $u,v\in \R^d$ the function
$$\R \ni s\ \longmapsto \ P(\underline{\varphi})(us+v)$$ is analytical and convex. We deduce from this that the function $$\underline{s}\ \longmapsto \ P(\underline{\varphi})(\underline{s})$$ is infinitely differentiable and convex on $\R^d$. We can prove that $P(\underline{\varphi})(\underline{s})$ is indeed analytical by the same argument used to prove the analyticity of $P_{\varphi}(s)$.

Similarly, we define  the level sets $E(\underline{\alpha})$ $(\underline{\alpha}\in \R^d)$ of $\underline{\varphi}$. A vector version of Theorem \ref{thm principal} is stated by just replacing the  derivative of the pressure function by gradient.

\bigskip
We finish the paper with two problems.
\bigskip

{\em Subshifts of finite type.} Our study is strictly restricted to the full shift dynamics. It is a challenging problem
to study the dynamics of subshift of finite type.

More general are dynamics with Markov property. More efforts are needed to deal with $\beta$-shift which are not Markovian.  New ideas are needed
to deal with these dynamics.

\bigskip
{\em Nonlinear cookie cutter.} The full shift is essentially the doubling dynamics $Tx = 2 x$ $\mod 1$ on the interval $[0,1)$.
Cookie cutters are the first interval maps coming into the mind after the doubling map. If the cookie cutter maps are not linear, it is a difficult problem.

Based on the computation made in \cite{PS}, Liao and Rams \cite{LR} considered a special piecewise linear map of two branches defined on two intervals
$I_0$ and $I_1$ and studied the following  limit
$$
          \lim_{n \to \infty} \frac{1}{n} \sum_{k=1}^n 1_{I_1}(T^kx)1_{I_1}(T^{2k}x).
$$
The techniques presented in the present paper can be used to treat the problem
for general piecewise linear cookie cutter dynamics  \cite{FLW,Wu}.

\end{document}